\documentclass[11pt, english,english]{scrartcl}
\usepackage[applemac]{inputenc}
\usepackage{graphicx,epsfig,color}
\usepackage{mathtools,mathrsfs}
\usepackage{amsmath,amssymb,amsthm}

\usepackage{bookmark}
\usepackage{comment}
\excludecomment{extra}
\excludecomment{k}
\includecomment{krad}
\usepackage{hyperref}

\newcommand{\C}{ \mathbb{C}}
\newcommand{\R}{ \mathbb{R}}
\newcommand{\Z}{ \mathbb{Z}}

\newcommand{\N}{ \mathbb{N}}

\newcommand{\eqdef}{\coloneqq}
\newcommand{\Top}{\mathsf{T}}

\newcommand{\Dop}{{\text D}}

\newcommand{\dint}{{\mathrm d}}

\newcommand{\Fourier}{ \mathcal{F}} 

\newcommand{\bx}{{\boldsymbol x}}
\newcommand{\bw}{{\boldsymbol \omega}}

\newcommand{\bk}{{\mathbf k}}

\DeclareMathOperator*{\esssup}{ess\,sup}

\DeclareMathOperator*{\spann}{span}
\DeclareMathOperator*{\argmin}{arg\,min}

\def\V#1{{\mathbf{#1}}}         
\def\Spc#1{{\mathcal{#1}}}  
\def\Op#1{{\mathrm{#1}}}  
\def\ee{\text{e}} 
\def\jj{\mathrm{j}}

\def\Proj{\mathrm{Proj}} 
\def\Identity{\mathrm{Id}} %

\newcommand{\embedIso}{\xhookrightarrow{\mbox{\tiny \text{iso.}}}}

\renewcommand{\[}{\begin{equation}}
\renewcommand{\]}[1]{\label{eq:#1}\end{equation}}

\newtheorem{definition}{Definition}
\newtheorem{proposition}{Proposition}
\newtheorem{corollary}{Corollary}

\newtheorem{lemma}{Lemma}
\newtheorem{theorem}{Theorem}

\newtheorem{example}{Example}
\newtheorem{remark}{Remark}

\title{Explicit Representations for Banach Subspaces of Lizorkin Distributions}

\author{
Sebastian Neumayer\footnotemark[1] \and Michael Unser\footnotemark[1]}
\begin{document}

\maketitle

\renewcommand*{\thefootnote}{\fnsymbol{footnote}}
\footnotetext[1]{Biomedical Imaging Group, \'Ecole polytechnique f\'ed\'erale de Lausanne (EPFL),
	Station 17, CH-1015 Lausanne, {\text \{firstname.lastname\}@epfl.ch}. }
\renewcommand*{\thefootnote}{\arabic{footnote}}




%
\begin{abstract}
	\textbf{Abstract} The Lizorkin space is well suited to the study of operators like fractional Laplacians and the Radon transform.
	In this paper, we show that the space is unfortunately not complemented in the Schwartz space.
	In return, we show that it is dense in $C_0(\R^d)$, a property that is shared by the larger Schwartz space  and that turns out to be useful for applications.
	Based on this result, we investigate subspaces of Lizorkin distributions that are Banach spaces and for which a continuous representation operator exists.
	Then, we introduce a variational framework that involves these spaces and that makes use of the constructed operator.
	By investigating two particular cases of this framework, we are able to strengthen existing results for fractional splines and 2-layer ReLU networks.
	
\end{abstract}


%

\section{Introduction}
This paper pertains to the Lizorkin space $\Spc S_\mathrm{Liz}(\R^d)$, which was first introduced in \cite{Lizorkin1963} for the investigation of partial differential equations.
It consists of the Schwartz functions $\Spc S(\R^d)$ for which all moments vanish.
For detailed expositions on the topic, we refer to \cite{Rubin1996,Samko2002}.
Surprisingly, $\Spc S_\mathrm{Liz}(\R^d)$ is still rather large, its closure under the $L_\infty$-norm being the space $C_0(\R^d)$.
Another attractive feature of $\Spc S_\mathrm{Liz}(\R^d)$ is that many non-invertible operators become invertible if their domain is restricted  to $\Spc S_\mathrm{Liz}(\R^d)$, which happen with the Radon transform and with fractional Laplacians.
This makes $\Spc S_\mathrm{Liz}(\R^d)$ well suited to theoretical analyses.
In the present paper, we show that $\Spc S_\mathrm{Liz}(\R^d)$ cannot be complemented in $\Spc S(\R^d)$.
In particular,  continuous projections onto $\Spc S_\mathrm{Liz}(\R^d)$ cannot exist.
This result is in sharp contrast to the periodic setting, where a projection actually exists.

The corresponding dual space $\Spc S^\prime_\mathrm{Liz}(\R^d)$ of Lizorkin distributions \cite{Yuan2010} is fairly large and has attracted increased interest over the past years (e.g., shearlet transform \cite{Bartolucci2020}, ridegelet transform  \cite{Kostadinova2014}, choice of activation functions in neural networks \cite{Sonoda2017}).
It is well known that $\Spc S^\prime_\mathrm{Liz}(\R^d)$ can be identified as the quotient space $\Spc S^\prime(\R^d)/\Spc P(\R^d)$, where $\Spc P(\R^d)$ denotes the space of polynomials.
Naturally, this leads to a representation problem if we want to make computations explicit, an issue that has not been addressed so far.
Here, our results directly imply that no continuous linear projector for the assignment of representatives can exist.
At first glance, this result appears to be quite discouraging as it implies that, in general, it is necessary to work with equivalence classes.
Fortunately, this can be circumvented if we consider appropriate Banach subspaces of $\Spc S^\prime_\mathrm{Liz}(\R^d)$.
Indeed, due to the density of $\Spc S_\mathrm{Liz}(\R^d)$ in $C_0(\R^d)$ and due to the Riesz theorem, the space of Radon measures is an embedded subspace of  $\Spc S^\prime_\mathrm{Liz}(\R^d)$ for which unique representations exist.
Further, we are able to provide a positive answer for more general cases if we restrict ourselves to subspaces  that can be equipped with a specific Banach-space structure.
In this case, we are able to provide a continuous representation operator for which the point evaluations are weak*-continuous.

The procedure to obtain these subspaces and the representatives is as follows:
Given a well understood pair of Banach spaces $(\Spc X, \Spc X^\prime)$, we construct Banach spaces $(\Spc X_{\Op T}, \Spc X_{\Op T}^\prime)$ through a linear homeomorphism $\Op T \colon \Spc S_1 \subset \Spc X \to \Spc S_2$ defined on some dense subspace of $\Spc X$.
The core advantage of our construction is that many properties carry over directly to $(\Spc X_{\Op T}, \Spc X_{\Op T}^\prime)$; for instance the set of the extreme points of $\Spc X_{\Op T}^\prime$ can be specified easily from those of $\Spc X^\prime$.
While our construction is abstract, it allows us to make use of the fact that many (differential) operators are homeomorphisms on $\Spc S_\mathrm{Liz}(\R^d)$ due to the density of  $\Spc S_\mathrm{Liz}(\R^d)$ in $\Spc X = C_0(\R^d)$.
Unfortunately, the space $\Spc X_{\Op T}^\prime$ usually still consists of mere equivalence classes.
Therefore, as second step, we formulate conditions under which we can identify the elements of $\Spc X_{\Op T}^\prime$  using a representation operator.
These conditions are fulfilled when the Green's function of the operator $\Op T$ is sufficiently regular.
Overall, this framework enables us to design a rich class of interesting new norms for which the related Banach subspaces of Lizorkin distributions have a continuous representation operator.
This makes our framework usable for applications.

Within the proposed setting, we study variational problems involving the constructed Banach spaces and the general representer theorems established in \cite{Unser2021,Unser2022}.
The fact that our abstract formulation involves spaces whose elements are equivalence classes  can be circumvented by the application of our representation operator.
We investigate two special cases for which the formulations become explicit.
First, we revisit fractional splines in arbitrary dimensions, which have been investigated before in \cite{Fageot2020tv,Unser2000a,Unser2017}.
These splines are a generalization of the traditional polynomial splines \cite{Schoenberg1946} and preserve most of their properties.
Note that (fractional) splines are still fashionable and recently found their way into neural network research \cite{Bohra2020,Fey2018,Parhi2021b}.
Although this is not included in our discussion, we may also use the model to study polynomial splines.
Overall, our approach leads to a unifying setting that includes a straightforward extension to the multivariate case.

As second example, we strengthen the representation results  for 2-layer ReLU networks established by Parhi and Nowak \cite{Parhi2021} and Bartolucci et al.\ \cite{Bartolucci2021}, which builds up on the univariate case investigated in \cite{Savarese2019}.
The involved norm was also studied from a theoretical point of view in \cite{Ongie2020b}.
While proofs in these works rely on a general result on the existence of sparse solutions for variational problems by Bredies and Carioni \cite{Bredies2020}, we are additionally able to identify the solution set as being the weak* closure of certain sparse solutions.
Similarly to \cite{Bartolucci2021}, our construction is related to reproducing-kernel Banach spaces \cite{Combettes2018,Lin2019,Zhang2009} since our representation operator is constructed using a kernel and the point evaluations are continuous.
The key ingredient that enables us to strengthen the results of these prior works is that we actually construct a predual space for the optimization domain, which enables us to use our proposed variational framework.

 The paper is organized as follows: The necessary preliminaries are provided in Section~\ref{sec:prelim}.
Then, we proceed with a discussion of $\Spc S_\mathrm{Liz}(\R^d)$ in Section~\ref{sec:LizorkinSpaces} and show that a continuous projection onto $\Spc S_\mathrm{Liz}(\R^d)$ cannot exist.
This part is complemented with a short discussion of the periodic case.
In Section~\ref{sec:LizorkinRep}, we identify subspaces of Lizorkin distributions for which a continuous representation operator exists.
Next, we relate these subspaces to several interesting research questions in Section~\ref{sec:VariationalProb}.
As warm-up, we investigate the construction of periodic (fractional) splines in Section~\ref{sec:PeriodicLizSpline}.
Here, no representation mechanism is necessary as we can use the projector.
Then, we introduce in Section~\ref{sec:VarFramework} our general variational framework involving the constructed Banach spaces, for which we detail two specific cases:
The delicate case of non-periodic (fractional) splines in Section~\ref{sec:LizSplines}; and a representer theorem for 2-layer ReLU neural networks in Section~\ref{sec:RadonReLU}.
Finally, conclusions are drawn in Section~\ref{sec:Conclusions}.

\section{Mathematical Preliminaries}\label{sec:prelim}
\label{Sec:Notations}
In this paper, we consider functions $f \colon \R^d \to \C$.
To describe their partial derivatives, we use the multi-index 
$\V k=(k_1,\dots,k_d) \in \N^d$ with the notational conventions $\V k!=\prod_{n=1}^d k_n!$, $|\V k|=k_1+\cdots+k_d$, $\V x^{\V k}=\prod_{n=1}^d x_n^{k_n}$ for any $\V x \in \R^d$, and
\begin{equation}
	\partial^{\V k} f(\V x)=\frac{\partial^{|\V k|}f(x_1,\dots,x_d)}{\partial^{k_1}_{x_1} \cdots \partial^{k_d}_{x_d}}.
\end{equation}
This enables us to write the multidimensional Taylor expansion around $\V x_0$ of an analytical function $f\colon \R^d \to \C$ in compact form as
\begin{equation}
f(\V x)=\sum_{n=0}^\infty \sum_{|\V k|=n} \frac{\partial^\V k f(\V x_0)}{\V k! } (\V x- \V x_0)^\V k.
\end{equation}

The Schwartz space \cite{Schwartz1966} of smooth and rapidly decreasing functions $\varphi\colon \R^d \to \C$ equipped with the usual Fr\'echet-Schwartz topology is denoted by $\Spc S(\R^d)$.
This space is an algebra for the multiplication as well as the convolution product.
Additionally, it is closed under translation, differentiation, and multiplication by polynomials.
Its continuous dual is the space of tempered distributions $\Spc S'(\R^d)$.
Moreover, $\Spc S^\prime(\R^d)$ (as well as $\Spc S(\R^d)$) is a nuclear Montel space, where a sequence in $\Spc S^\prime(\R^d)$ converges with respect to the strong dual topology if and only if it converges in the weak* topology.
Therefore, it does not actually matter which of the two topologies we choose for $\Spc S^\prime(\R^d)$. 
The Montel property also implies that $\Spc S(\R^d)$ is reflexive, which means that there exists an isomorphism between the topological vector spaces $\Spc S^{\prime\prime}(\R^d)$ and $\Spc S(\R^d)$.
Note that the Lebesgue spaces $L_p(\R^d)$ for $p\in[1,\infty)$ are the completion of the set $\Spc S(\R^d)$ under the $L_p$-norm $\|\cdot\|_{L_p}$.
For $p=\infty$, we have that \smash{$\overline{(\Spc S(\R^d),\|\cdot\|_{L_\infty})}=C_0(\R^d)$}, namely, the space of continuous functions that vanish at infinity.
The dual of $C_0(\R^d)$ is the space $\Spc M(\R^d)=\{f \in \Spc S'(\R^d): \|f\|_{\Spc M}<\infty\}$ of bounded 
Radon measures with norm
\begin{equation}
\|f\|_{\Spc M}=\sup_{\varphi \in \Spc S(\R^d): \|\varphi\|_{L_\infty}\le 1} \langle f, \varphi \rangle.
\end{equation}
The latter is an isometrically-embedded superset of $L_1(\R^d)$, which implies that  $\|f\|_{L_1}= \|f\|_{\Spc M}$ for all $f \in L_1(\R^d)$.
We also need the weighted Lebesgue space $L_{\infty,\alpha}(\R^d)$, $\alpha \geq 0$,  defined via the weighted norm $\Vert f \Vert_{\infty,\alpha} \coloneqq \esssup_{\V x \in \R^d} \vert f(\V x) \vert (1 + \Vert \V x \Vert_2)^{-\alpha}$, which consists of functions that grow with order at most $\alpha$.

The Fourier transform $\Fourier \colon L_1(\R^d) \to C_0(\R^d)$ of a function $\varphi \in L_1(\R^d)$ is defined as
\begin{align}
\widehat \varphi(\bw)\coloneqq \Fourier\{\varphi\}(\bw)=\frac{1}{(2 \pi)^d} \int_{\R^d} \varphi(\V x) \ee^{-\jj \langle \bw, \V x \rangle} \dint \V x.
\end{align}
As the Fourier transform $\Fourier \colon  \Spc S(\R^d) \to \Spc S(\R^d)$ is an isomorphism, it can be extended by duality to $\Spc S'(\R^d)$.
Specifically, $\widehat f \in \Spc S'(\R^d)$ is the (unique) {\em generalized Fourier transform} of $f\in \Spc S'(\R^d)$ if and only if $\langle \widehat f, \varphi \rangle=\langle f, \widehat\varphi \rangle$ for all $\varphi  \in \Spc S(\R^d)$.
Finally, we note that the analytic Schwartz functions form a dense subset of $\Spc S(\R^d)$, which can be seen as follows.
As the smooth and compactly supported functions $\Spc D(\R^d)$ are dense in $\Spc S(\R^d)$, we also get that $\Fourier (\Spc D(\R^d))$ is dense in $\Spc S(\R^d)$.
Due to the Paley-Wiener theorem, the Fourier transform of any $f \in \Spc D(\R^d)$ is analytic and also entire.
Hence, these functions are dense.

The simplest way to specify fractional derivatives or integrals is to describe their action in the Fourier domain. 
Let us start with $d=1$. The one-dimensional fractional derivative $\Op D^\alpha\colon \Spc S(\R) \to \Spc S^\prime(\R)$ of order $\alpha\ge0$ is defined as
\begin{align}
	\label{Eq:FracDeriv}
	\Op D^{\alpha}\{\varphi\}(t) =\Fourier^{-1}\bigl\{(\jj \cdot)^{\alpha} \widehat \varphi\bigr\}(t).
\end{align}
For $\alpha=n\in \N$, $\Op D^n=\frac{\dint^n}{\dint t^n}$ coincides with the classical $n$th order derivative.
Definition~\eqref{Eq:FracDeriv} is also valid for negative orders, in which case it yields a fractional integral \cite{Unser2000a}.
In fact, the impulse response of $\Dop^{-\alpha}$
is the Green's function of $\Dop^\alpha$, which is given by
\begin{align}
	\rho_{\alpha}(t)=\Fourier^{-1}\left\{\frac{1}{(\jj \cdot)^{-\alpha}} \right\}(t)=\begin{cases}
		\frac{t_+^{\alpha-1}}{\Gamma(\alpha)},& \alpha-1\in \R^{+} \backslash \N\\
		\tfrac{\mathrm{sgn}(t)}{2} \frac{t^{n}}{ n!},& \alpha-1=n\in \N.
	\end{cases}
\end{align}
Likewise, the fractional Laplacian $(-\Delta)^{\alpha/2}$ of order $\alpha\in(1,\infty)$ is the linear-shift-invariant operator (LSI) whose frequency response is $\|\bw\|^{\alpha}$. Its inverse is the fractional integrator $(-\Delta)^{-\alpha/2}$, which corresponds to a frequency-domain multiplication by $\|\bw\|^{-\alpha}$.
Fractional derivatives and Laplacians are part of the same family of operators (isotropic LSI and scale-invariant) with their distributional impulse response for $\alpha>d$ being given by
\begin{align}
	k_{\alpha,d}(\V x)=\Fourier^{-1}\left\{\frac{1}{\|\cdot\|^{\alpha} } \right\}(\V x)=\begin{cases}
		(-\Delta)^{-n} \{\delta\}, & \alpha/2=n \in \N\\
		B_{n,d} \; \|\V x\|^{2n} \log(\|\V x\|) , & \alpha-d=2n \in 2\N\\
		A_{\alpha,d} \; \|\V x\|^{\alpha-d},& \alpha-d \notin 2\N
	\end{cases}
	\label{Eq:GreenLap}
\end{align}
with 
constants
$A_{d,\alpha}=
\frac{\Gamma( (d-\alpha)/2)}{2^\alpha \pi^{d/2} \Gamma(\alpha/2)}$
and \smash{$B_{d,n}=
\frac{(-1)^{1+n} }
{2^{2n+d-1}\pi^{d/2}  \Gamma( n+d/2) n! }$}.
For a more detailed exposition on the topic, we refer to \cite{GelfandShilov1964,Samko1993,Stinga2019}.


\section{Lizorkin Spaces}\label{sec:LizorkinSpaces}
The Lizorkin space $\Spc S_\mathrm{Liz}(\R^d)$ is the closed subspace of $\Spc S(\R^d)$ that consists of the functions whose moments
of any order $\V k$ are zero, so that
\begin{align}
\Spc S_\mathrm{Liz}(\R^d)=\Bigl\{\varphi \in \Spc S(\R^d): \int_{\R^d} \V x^\bk \varphi(\V x) \dint \V x=0, \forall \V k\in \N^d\Bigr\}.
\end{align}
A nice overview with properties of $\Spc S_\mathrm{Liz}(\R^d)$ is given in \cite{Troyanov2009}.
Equivalently, we can describe these functions in the Fourier domain through
\[\widehat{\Spc S}_\mathrm{Liz} (\R^d) = \Fourier\bigl(\Spc S_\mathrm{Liz} (\R^d)\bigr) = \bigl\{\psi \in \Spc S(\R^d) : \partial^{\V k} \psi(\V 0) = 0 \quad \forall \V k  \in \N^d\bigr\}.\]{a}
Although closed subspaces of reflexive topological vector spaces are in general not reflexive, this property holds for Fr\'echet spaces.
Hence, the spaces $\Spc S_\mathrm{Liz}(\R^d)$ and $\widehat{\Spc S}_\mathrm{Liz} (\R^d)$ are reflexive.
Further, we have for all $\varphi \in \Spc S_\mathrm{Liz}(\R^d)$, $\V x_0 \in \R^d$ and $a \in \R$ that $\varphi(\cdot-\V x_0) \in \Spc S_\mathrm{Liz}(\R^d)$ and $\varphi(\cdot/a) \in \Spc S_\mathrm{Liz}(\R^d)$.
Finally, we note that $\Spc S_\mathrm{Liz}(\R^d) \cap \Spc D(\R^d)=\{0\}$. Indeed, if $\varphi \in \Spc D(\R^d)$, then 
$\widehat \varphi$ is entire and hence equal to its Maclaurin expansion.
But if 
$\varphi \in \Spc S_\mathrm{Liz}(\R^d)$, then the Taylor series of $\widehat \varphi$ is $0$.

We are going to show that $\Spc S_\mathrm{Liz}(\R^d)$ cannot be complemented in $\Spc S(\R^d)$; in other words,
a continuous projector $\Op P_\mathrm{Liz}\colon  \Spc S(\R^d) \to \Spc S_\mathrm{Liz}(\R^d)$ cannot exist.
Before we prove this negative result, we first discuss the easier case of periodic Lizorkin functions, for which a continuous  projection actually exists.

\subsection{Periodic Lizorkin Spaces}
The functions of interest are $T$-periodic and typically specified only over their main period $\mathbb{T}=[0,T]$.
The corresponding space of test functions is \smash{$\Spc S(\mathbb{T})=C^{\infty}_\textup{perio}(\mathbb{T})$}, which is in one-to-one correspondence with the Fréchet space of rapidly decaying sequences\footnote{This space is denoted by ``$s$'' in \cite{Grothendieck1955}. It is the discrete analog of the Schwartz space $\Spc S(\R)$.} $\Spc S(\Z)$ via the Fourier homeomorphism \cite{Fageot2019,Treves2006}.
More precisely, there are Fourier coefficients $\widehat \varphi[\cdot] \in \Spc S(\Z)$ such that
\begin{align}
	\varphi(t)=\sum_{n \in \Z} \widehat \varphi[n]\ee^{\jj n \omega_0 t} \in \Spc S(\mathbb{T})
\end{align}
with $\omega_0=\frac{2 \pi}{T}$.
This expansion of $\varphi$ is unique and $\widehat \varphi[n]= \langle \varphi, \ee^{-\jj n \omega_0 \cdot}\rangle_\mathbb{T}$, where $ \langle f, g  \rangle_\mathbb{T}  = \frac{1}{T}\int_\mathbb{T} f(t) g(t)\dint t$. The continuous dual of 
$\Spc S(\mathbb{T})$ is the space of periodic distributions $\Spc S'(\mathbb{T})=\Spc S'_\textup{perio}(\R)$, which is itself homeomorphic to the space $\Spc S'(\Z)$ of slowly growing sequences.
Indeed, it holds \smash{$f \in \Spc S'(\mathbb{T}) \Leftrightarrow \widehat f[\cdot]\in \Spc S'(\Z)$}, where \smash{$\widehat f[n]$} denotes the $n$-th Fourier coefficient of $f$.

To ensure invertibility of the continuous fractional-derivative operator $\Op D^{\alpha}\colon \Spc S(\mathbb{T}) \to \Spc S(\mathbb{T})$ given by
\begin{align}
	\label{Eq:DalphaPerio}
	\Op D^{\alpha}\{\varphi\}(t) =\sum_{n\in \Z} (\jj \omega_0 n)^{\alpha} \widehat \varphi[n] \ee^{\jj n \omega_0 t},
\end{align}
we restrict ourselves to the subspace
\begin{align}
	\label{Eq:LizPerio}
	\Spc S_0(\mathbb{T})=\Bigl \{\varphi \in \Spc S(\mathbb{T}): \langle 1, \varphi  \rangle_\mathbb{T}  =0\Bigr\},
\end{align}
which inherits the nuclear topology from $\Spc S(\mathbb{T})$. While \eqref{Eq:LizPerio}
imposes a restriction only on the mean value of $\varphi$, the resulting space $\Spc S_0(\mathbb{T})$ is the proper periodic counterpart of the $\Spc S_\mathrm{Liz}(\R)$,  since the only periodic polynomials are constants.
The periodic setting is simple, in that $\Spc S_0(\mathbb{T})$ is 1-complemented in $\Spc S(\mathbb{T})$ with $\Spc S(\mathbb{T})=\Spc S_0(\mathbb{T}) \oplus \Spc P_0$ and
\begin{align}
	\Spc P_0=\{b_0 \cdot 1: b_0 \in \R\} \subset \Spc S(\mathbb{T}).
\end{align}
Correspondingly, we introduce the continuous projection $\Op P_0\colon \Spc S(\mathbb{T}) \to \Spc S_0(\mathbb{T})$ with
\begin{align}
	\label{Eq:ProjPerio}
	\Op P_0\{\phi\}=\varphi - \langle 1, \varphi  \rangle_\mathbb{T} 1.
\end{align}
As $\Spc S_0(\mathbb{T})=\Op P_0( \Spc S(\mathbb{T}))$, its dual is
$\Spc S'_0(\mathbb{T})=\Op P^\ast_0( \Spc S'(\mathbb{T}))$ with $\Spc S'(\mathbb{T})=\Spc S'_0(\mathbb{T}) \oplus \Spc P_0$.
Here, we can identify $\Spc P_0'=\Spc P_0$ because the space is spanned by $1 \in \Spc S(\mathbb{T})\subset \Spc S'(\mathbb{T})$ with $\langle 1,1\rangle_\mathbb{T}=1$.
The latter property also implies that $\Op P^\ast_0=\Op P_0$, which makes the projection~\eqref{Eq:ProjPerio} also applicable to periodic distributions.

As the space $\Spc P_0$ of constant polynomials is indeed  the null space of $\Dop^{\alpha}\colon \Spc S(\mathbb{T}) \to \Spc S(\mathbb{T})$ with $\alpha>0$, we can restrict $\Dop^{\alpha}$ to
a homeomorphism $\Op D^\alpha\colon \Spc S_0(\mathbb{T})\to \Spc S_0(\mathbb{T})$ for any $\alpha\in \R$.
By duality, the same holds true on $\Spc S'_0(\mathbb{T})$ with the Fourier-domain definition
\eqref{Eq:DalphaPerio} of $\Dop^{\alpha}$ being applicable to periodic distributions as well.
In particular, the fractional integrator $\Op D^{-\alpha}\colon \Spc S'_0(\mathbb{T}) \to \Spc S'_0(\mathbb{T})$ of order $\alpha\ge 0$ is given by
\begin{align}
	\label{Eq:FracInt}
	\Op D^{-\alpha}\{f\}(t) =\sum_{n\in \Z\backslash\{0\}} \frac{1}{(\jj n \omega_0)^{\alpha}} \widehat f[n] \ee^{\jj n \omega_0  t}.
\end{align}

\subsection{Nonexistence of a Continuous Projector \texorpdfstring{$\Op P_\mathrm{Liz}\colon \Spc S(\R^d) \to \Spc S_\mathrm{Liz} (\R^d)$}{P\_Liz: S(Rd) -> S\_Liz(Rd)}}
\label{sec:NonExistence}
To prove the nonexistence of a continuous linear projection onto the Lizorkin space $\Spc S_\mathrm{Liz} (\R^d)$, we first show that the closed set $\mathcal P(\R^d)$ is not complemented in $\Spc S^\prime (\R^d)$.
\begin{theorem}\label{thm:exist_proj_dual}
	There exists no topological complement of $\mathcal P(\R^d)$ in $\Spc S^\prime (\R^d)$.
\end{theorem}
\begin{proof}
	Assume there is a complement.
	In other words, assume that a continuous projector $ \Op P \colon \Spc S^\prime (\R^d) \to \mathcal P(\R^d) \subset \Spc S^\prime (\R^d)$ exists.
	We consider $\widehat{\delta}_{\V x_0} \in \Spc S^\prime (\R^d)$, ${\V x_0} \in \R^d$,
	and set $p_{\V x_0} \coloneqq \Op P \{\widehat{\delta}_{\V x_0}\}\in \mathcal P(\R^d)$.
	Using relations such as
	\[\lim_{t \to 0} \frac{\delta_{te_k} - \delta}{t} = - \nabla_{e_k} \delta\]{b}
	and similar ones for higher-order derivatives, we observe that \smash{$\mathcal P(\R^d) \subset \overline{\spann\{\widehat{\delta}_{\V x_0}\}_{\Vert {\V x_0} \Vert \leq 1}}$}, where all equalities are in the sense of distributions.
	Since $\Op P$ is a continuous projection onto $\mathcal P(\R^d)$, this implies that
	\[\mathcal P(\R^d) \subset \overline{\spann\{p_{\V x_0}\}_{\Vert {\V x_0} \Vert \leq 1}}.\]{c}
	In the sequel, we show that the polynomials $p_{\V x_0}$ with $\Vert {\V x_0} \Vert_2 \leq 1$ have a common maximum degree $m$, which results in the contradiction that
	$\mathcal P(\R^d) \subset \mathcal P_m(\R^d)$.

	If no common maximum exists, then there is a sequence $ \{{\V h}_n\}_{n \in \N} \in \R^d$ with  $\Vert {\V h}_n \Vert_2 \leq 1$ such that $\{p_{{\V h}_n}\}_{n \in \N} $ is a sequence of polynomials with unbounded degree.
	By passing to a subsequence, we can assume that ${\V h}_n \to {\V h}$ for some ${\V h} \in \R^d$ with $\Vert {\V h} \Vert \leq 1$.
	 Due to the continuity of $\Op P$ and $\mathcal F$, we also have that $\widehat{p}_{{\V h}_n} \to \widehat{p}_{\V h}$ in the sense of distributions.
	 Setting $p_{{\V h}_n} =  \sum_{j=0}^{m_n} a_{j,n} x^j$ with $m_n \to \infty$ and $a_{m_n,n} \neq 0$, this can be written as
	 \[\widehat{p}_{{\V h}_n} = \sum_{j = 0}^{m_n} (-2\pi \jj)^{j} a_{j,n} \frac{\partial^j}{\partial \xi^j } \delta_0 \to \widehat{p}_{\V h}.\]{d}
	 Dropping again to a subsequence, we assume that $m_n$ is monotonically increasing.
	 Using Borel's theorem, we then pick $\varphi \in \Spc S (\R^d)$  with $\frac{\partial^{m_n}}{\partial \xi^{m_n}} \varphi(0) = (a_{{m_n},n})^{-1} C_n$, where $C_n$ is chosen such that
	 \[\biggl \vert \sum_{j = 0}^{m_n} (-2\pi \jj)^{j} a_{j,n} \frac{\partial^j}{\partial \xi^j } \varphi(0) \biggr \vert \geq n.\]{e}
	 Hence, $\widehat{p}_{{\V h}_n}(\varphi) \to \infty$, which contradicts that $\widehat{p}_{{\V h}_n} \to \widehat{p}_{\V h} \in \Spc S^\prime (\R^d)$.
	 Consequently, all $p_{\V h}$ with $\Vert {\V h} \Vert \leq 1$ have a common maximum degree $m$.
\end{proof}
\begin{remark}
	Theorem~\ref{thm:exist_proj_dual} implies that there is no continuous linear projection $\Op P\colon \Spc S^\prime(\R^d) \to \Spc S^\prime(\R^d)$ with $\ker \Op P = \mathcal P(\R^d)$.
	Otherwise, we would have that $(\mathrm{Id} - \Op P)$ is a continuous projector onto $\mathcal P(\R^d)$.
	In particular, representatives of Lizorkin distributions cannot be assigned in a continuous linear way.
\end{remark}
Now, the desired nonexistence result follows immediately.
\begin{corollary}\label{thm:exist_proj}
	There exists no topological complement of $\Spc S_\mathrm{Liz} (\R^d)$ in $\Spc S (\R^d)$.
\end{corollary}
\begin{proof}
On the contrary, let us assume that a continuous linear projection $\Op P_\mathrm{Liz}  \colon \Spc S(\R^d) \to \Spc S_\mathrm{Liz}(\R^d)$ exists.
	Then, the adjoint map $\Op P_\mathrm{Liz}^*  \colon \Spc S^\prime(\R^d) \to \Spc S^\prime(\R^d)$ is a projection as well.
	Due to the fact that
	\begin{equation}
		\langle \Op P_\mathrm{Liz}^* \{f\}, \varphi \rangle = \langle f, \Op P_\mathrm{Liz}\{\varphi\} \rangle
	\end{equation}
	for all $f \in \Spc S^\prime(\R^d) $ and $\varphi \in \Spc S(\R^d)$, its null space is given by $\ker \Op P_\mathrm{Liz}^* = \mathcal P(\R^d)$.
	Hence, $(\mathrm{Id} -\Op P_\mathrm{Liz}^*)$ would be a projection onto $\mathcal P(\R^d)$, which contradicts Theorem~\ref{thm:exist_proj_dual}.
\end{proof}

\subsection{Closure of the Lizorkin Space}\label{sec:CloseLiz}
Despite the negative finding of Section~\ref{sec:NonExistence}, we are nevertheless able to provide a result that is useful for applications.
\begin{theorem}\label{lem:DensLiz}
	It holds that $\overline{\bigl(\Spc S_\mathrm{Liz}(\R^d), \Vert \cdot \Vert_\infty\bigr)} = C_0(\R^d)$.
\end{theorem}
We note that the result was already mentioned in \cite{Samko1995}, but without a proof.
\begin{proof}
	Using the function $\tilde \varphi_0\colon \R^d \to \R$ with $\tilde \varphi_0(\V x) = \exp(-1/(1-(2\Vert \V x \Vert)^2))/n_d$ ($n_d$ is the normalizing constant) for $\Vert \V x \Vert < 1/2$ and zero else, we define $\varphi_0 \colon \R^d \to [0,1]$  via $\varphi(\V x) = (\chi_{B_{1}(0)} \ast \tilde \varphi_0 ) (2\V x)$ with $\chi_{B_{1}(0)}$ being the characteristic function of the unit ball.
	This  function is smooth, symmetric, and satisfies  that $\varphi(\V x) = 1$ for $\vert \V x \vert \leq 1/4$ and $\varphi(\V x) = 0$ for $\vert \V x \vert \geq 3/4$.
	Based on this function, we define
	\[\phi_{\V n} = \frac{\V x^{\V n}}{\V n!}\varphi_0 \in \Spc S(\R^d),\quad \V n \in \N^d\]{f}
	and set
	\[E_2  = \overline{\spann\{\phi_{\V n}: \V n \in \N^d\}} \subset \Spc S(\R^d).\]{g}
	Next, we observe that  $(\widehat \phi_{\V n}  - t^{-\vert \V n \vert -d} \widehat \phi_{\V n} (\cdot/t)) \to \widehat \phi_{\V n}  \in C_0(\R^d)$ as $ t \to \infty$, where the sequences have all moments equal to zero.
	Hence, we have that
	\begin{equation}
		\widehat {E}_2 \in \overline{\bigl(\Spc S_\mathrm{Liz}(\R^d), \Vert \cdot \Vert_\infty\bigr)}.
	\end{equation}
	
	To conclude the argument, we show  that $\widehat{\Spc S}_\mathrm{Liz}(\R^d) + E_2$ contains the entire Schwartz functions, so that its closure under the Schwartz topology is already the complete space $\Spc S(\R^d)$.
	Then, the same holds true for $\Spc S_\mathrm{Liz}(\R^d) + \widehat {E}_2$ and, consequently, we get that
	\begin{equation}
		\overline{\bigl(\Spc S_\mathrm{Liz}(\R^d), \Vert \cdot \Vert_\infty\bigr)} = \overline{\bigl(\Spc S_\mathrm{Liz}(\R^d) + \widehat{E}_2, \Vert \cdot \Vert_\infty\bigr)} = \overline{\bigl(\Spc S (\R^d), \Vert \cdot \Vert_\infty\bigr)} = C_0(\R^d).
	\end{equation}
	The Taylor series of any entire function $f$ converges absolutely for any $\V x \in \R^d$.
	It then holds that
	\[g = f\varphi_0= \sum_{\V n \in \N^d} \partial^{\V n} f(\V 0) \phi_{\V n} \in \Spc S(\R^d).\]{h}
	Hence, we get that $(f - g) \in \widehat{\Spc S}_\mathrm{Liz}(\R^d)$. It remains to show that $g \in E_2$ or, equivalently, that
	\[g_j = \sum_{\V n \in \N^d, \vert \V n \vert < j} \partial^{\V n} f(\V 0) \phi_{\V n}\to g\]{i}
	in the Schwartz topology.
	For any ${\V \alpha}, {\V k}  \in \N^d$, it holds that
	\begin{align}
		\Vert \V  x^{\V \alpha} \partial^{\V k}  (g-g_j)\Vert_\infty &\leq \sup_{\vert \V x \vert \leq 3/4} \vert \V x^{\V \alpha} \vert \sum_{\V n \in \N^d, \vert \V n \vert \geq j} \Bigl \vert \frac{\partial^{\V n} f(\V 0)}{\V n!} \partial^{\V k}  \bigl(\V x^{\V n} \varphi_0(\V x)\bigr)\Bigr \vert\notag\\
		& \leq \sup_{\vert \V x \vert \leq 3/4} \sum_{\V n \in \N^d, \vert \V n \vert \geq j} \vert \partial^{\V n} f(\V 0)\vert \sum_{\substack{{\V k} _1 \leq {\V k} ,\V n}} {{\V k }\choose{{\V k} _1}}\frac{\vert \V x ^{\V n- {\V k} _1}\vert}{(\V n-{\V k} _1)!} \vert \partial^{{\V k} -{\V k} _1} \varphi_0(\V x)\vert \notag\\
		&\leq C  \sum_{\V n \in \N^d, \vert \V n \vert \geq j} \vert \partial^{\V n} f(\V 0)\vert \sum_{\substack{{\V k} _1 \leq {\V k} , \V n}} {{{\V k} }\choose{{\V k} _1}}\frac{1}{(\V n-{\V k} _1)!}\notag\\
		&\leq C \sum_{\V n \in \N^d, \vert \V n \vert \geq j} \vert \partial^{\V n} f(\V 0)\vert \frac{\V n^{\V k} }{ \V n!}.
	\end{align}
	The last expression converges to zero as $j \to \infty$ if
	\[1<\limsup_{j \to \infty} \biggl(\sum_{\vert \V n\vert=j} \frac{\vert \partial^{\V n} f(\V 0)\vert}{\V n!}j^{\vert {\V k}  \vert} \biggr)^{-\frac{1}{j}}.\]{j}
	However, it holds that $\limsup_{j \to \infty} j^\frac{1}{j} = 1$ and, hence, the claim follows since the Taylor expansion converges absolutely for any $\V x \in \R^d$.
\end{proof}
By duality, Theorem~\ref{lem:DensLiz} implies that the Radon measures $\mathcal M(\R^d)$ are continuously embedded into the space $\Spc S^\prime_\mathrm{Liz}(\R^d)$ of Lizorkin distributions.

\begin{remark}
	In \cite{Samko1995}, it was shown that the same results hold for $L_p(\R^d)$ with $1 \leq p < \infty$.
\end{remark}

\section{Banach Subspaces of Lizorkin Distributions}\label{sec:LizorkinRep}
In contrast to the periodic case, the space $\Spc S^\prime_\mathrm{Liz}(\R^d)$  is an abstract space of equivalence classes, which means that the assignment of representatives for computational purposes is difficult.
Therefore, we want to restrict our attention to subspaces with more structure.
As our proposed framework is also applicable for other spaces, we outline it in full generality and explicitly provide the specifications for $ \Spc S_\mathrm{Liz}(\R^d)$  as discussion.

Let $\Spc S_1$, $\Spc S_2$ be two topological vector spaces with a linear homeomorphism $\Op T \colon \Spc S_1 \to \Spc S_2$ and let $\Op T^* \colon \Spc S_2^\prime \to \Spc S_1^\prime$ be defined via duality.
The simplest choice for the construction of Banach subspaces of $\Spc S^\prime_\mathrm{Liz}(\R^d)$ is $\Spc S_1 = \Spc S_2 = \Spc S_\mathrm{Liz}(\R^d)$, but variations of this setting are clearly possible.
In the Lizorkin setting, a possible choice of operator is the fractional Laplacian $\Op T = (-\Delta)^\alpha$, which is discussed as first example in Section~\ref{sec:Liz_Frac}.
Now, let $(\Spc X, \Spc X^\prime)$ be a dual pair of Banach spaces whose norm $\Vert \cdot \Vert_{\Spc X}$ is continuous w.r.t.\ the topology of $\Spc S_1$ such that \smash{$\overline{\bigl(\Spc S_1, \Vert \cdot \Vert_{\Spc X }\bigr)} = \Spc X$}.
For $\Spc S_1 = \Spc S_\mathrm{Liz}(\R^d)$, we have seen in Section~\ref{sec:CloseLiz} that $\Spc X = L_p(\R^d)$, $p \in [1,\infty)$, and $\Spc X = C_0(\R^d)$ are admissible choices, as their norms are indeed compatible with the Schwartz topology.
The density enables us to write that
\begin{equation}\label{eq:TV_T}
	\Vert f \Vert_{\Spc X^\prime} = \sup_{\varphi \in \Spc X: \|\varphi\|_{\Spc X}\le 1} \langle f, \varphi \rangle = \sup_{\varphi \in \Spc S_1: \|\varphi\|_{\Spc X}\le 1} \langle f, \varphi \rangle
\end{equation}
for any $f \in \Spc X^\prime$.
Given any  $f \in \Spc S_1^\prime$ for which \eqref{eq:TV_T} is finite, the bounded linear transformation (BLT) theorem implies that there exists a unique continuous extension to some element in $\Spc X^\prime$.
Conversely, any $f \in \Spc X^\prime$ defines a unique element in $ \Spc S_1^\prime$  via restriction due to the compatibility of the norm and the topology.

In this setting, we define the abstract space
\begin{align}
	\Spc X^\prime_{\Op T}=&\bigl\{f \in \Spc S'_2: \|\Op T^* \{f\}\|_{\Spc X^\prime}<\infty\bigr\}=\bigl\{\Op T^{-*} \{g\} \in \Spc S'_2:g\in \Spc X^\prime \bigr\},\label{eq:RepMeasGen} 
\end{align}
which is a Banach space if equipped with the norm \smash{$\Vert \cdot \Vert_{\Spc X_{\Op T}^\prime} \coloneqq \|\Op T^*\{\cdot\}\|_{\Spc X^\prime}$}.
In particular, by choosing $\Spc S_2 = \Spc S_\mathrm{Liz}(\R^d)$, we can construct a subspace of $ \Spc S^\prime_\mathrm{Liz}(\R^d)$ and equip it with a Banach-space structure.
The norm of \smash{$\Spc X^\prime_{\Op T}$} can be rewritten in dual form as
\begin{align}
	\| f \|_{\Spc X_{\Op T}^\prime} &= \sup_{\varphi \in \Spc S_1: \|\varphi\|_{\Spc X}\le 1} \langle \Op T^*\{f\}, \varphi \rangle = \sup_{\varphi \in \Spc S_2: \|\Op T^{-1}\{\varphi\}\|_{\Spc X}\le 1} \langle \Op T^*\{f\}, \Op T^{-1}\{\varphi\} \rangle\notag\\
	& = \sup_{\varphi \in \Spc S_2: \|\Op T^{-1}\{\varphi\}\|_{\Spc X}\le 1} \langle f, \Op T \Op T^{-1}\{\varphi\} \rangle = \sup_{\varphi \in \Spc S_2: \|\Op T^{-1}\{\varphi\}\|_{\Spc X}\le 1} \langle f, \varphi \rangle.
\end{align}
Consequently, from the BLT theorem again, any $f \in \Spc X^\prime_{\Op T}$ can be extended to a continuous functional with domain
\begin{equation}
	\Spc X_{\Op T} =  \overline{(\Spc S_2,\|\Op T^{-1}\{\cdot\}\|_{\Spc X})},
\end{equation}
which is identified as a predual of $\Spc X^\prime_{\Op T}$ since $\Op T^{-1}$ is continuous.
Likewise the operator $\Op T^{-1}$ can be extended to a continuous and surjective operator  $\overline{\Op T^{-1}} \colon \Spc X_{\Op T} \to \Spc X$.
Now, it holds that $\langle f, \varphi \rangle = \langle g, \overline{\Op T^{-1}} \{\varphi\} \rangle$ for any $f = \Op T^{-*} \{g\}$ and $\varphi \in \Spc X_{\Op T}$.
Hence, the weak* convergence of a sequence $f_n = \Op T^{-*} \{g_n\}$  to $f = \Op T^{-*} \{g\}$ is equivalent to the weak* convergence of $g_n$ to $g$.

Let us now discuss in more detail the case $\Spc S_1 = \Spc S_2 = \Spc S_\mathrm{Liz}(\R^d)$.
To simplify the notation, we stick to $\Spc X = C_0(\R^d)$, but the same argumentation applies to $L_p(\R^d)$.
For the remainder of this section, we use the more specific notations
\begin{align}
	\Spc  S'_{\mathrm{Liz},\Op T}(\R^d)=\bigl\{\Op T^{-*}  \{\mu\} \in \Spc S'_\mathrm{Liz}(\R^d) :\mu\in \Spc M(\R^d) \bigr\}
\end{align}
and $\Spc  S_{\mathrm{Liz},\Op T}(\R^d)$ for the dual and for the predual, respectively.
At first glance, the search for representatives for \smash{$\Spc  S'_{\mathrm{Liz},\Op T}(\R^d)$} is as difficult as before because we are still dealing with elements in $\Spc S'_\mathrm{Liz}(\R^d)$.
To resolve this issue, let us assume that there exist continuous elements $\rho_{\Op T,\V y} = \Op T^{-*}  \{ \delta(\cdot-\V y)\} \in C(\R^d)$ in each equivalence class.
Then, we conclude, for any $f = \Op T^{-*}  \{\mu\}  \in \Spc  S'_{\mathrm{Liz},\Op T}(\R^d)$ and $\varphi \in \Spc S_\mathrm{Liz}(\R^d)$, that
\begin{align}
	\bigl \langle {\Op T}^{-*} \{\mu\} , \varphi \bigr \rangle =& \bigl \langle \mu, \Op T^{-1} \{\varphi\}  \bigr \rangle \notag= \int_{\R^d} \Op T^{-1} \{\varphi\} (\V y) \dint \mu(\V y) = \int_{\R^d} \bigl \langle \Op T^{-*}  \{ \delta(\cdot-\V y)\}, \varphi \bigr \rangle \dint \mu(\V y)\\
	=& \int_{\R^d} \int_{\R^d} \bigl(\rho_{\Op T,\V y}(\V x) - p_{\V y}(\V x) \bigr) \varphi(\V x) \dint \V x \dint \mu(\V y),
\end{align}
where the polynomials $p_{\V y}(\V x) \in \Spc P(\R^d)$ are added to ensure the following properties:
First, the kernel $h(\V x, \V y) = \rho_{\Op T,\V y}(\V x)  - p_{\V y}(\V x)$ must be bi-continuous and bounded for some $g \in L_{\infty,\alpha}(\R^d)$, $\alpha \geq 0$, and every $\V y \in \R^d$ by
\begin{align}
	\vert h(\V x, \V y)\vert &\leq g(\Vert \V x \Vert).
\end{align}
Second, we require that $h(\V x, \V y) \to 0$ for a given $\V x \in \R^d$ and $\Vert \V y \Vert \to \infty$.
Using Fubini's theorem and the growth control, we can then identify the distribution $\Op T^{-*}  \{\mu\}  \in \Spc S'_\mathrm{Liz}(\R^d)$ as the continuous function $f(\V x) = \int_{\R^d}  \rho_{\Op T,\V y}(\V x) - p_{\V y}(\V x) \dint \mu(\V y)$.
Due to the growth bound, this function corresponds to a unique distribution in $\Spc S'(\R^d)$.
We collect these observations together with a few properties in Theorem~\ref{thm:RepLiz}.
\begin{theorem}\label{thm:RepLiz}
	Assume that the Schwartz kernel $h(\V x, \V y) = \rho_{\Op T,\V y}(\V x)  - p_{\V y}(\V x)$ is bi-continuous and  bounded for every $\V y \in \R^d$ as
	\begin{align}
		\vert h(\V x, \V y)\vert &\leq g(\Vert \V x \Vert)
	\end{align}
for some $g \in L_{\infty,\alpha}(\R^d)$, $\alpha \geq 0$.
	Then, any element  $f = \Op T^{-*}  \{\mu\} \in\Spc S^\prime_{\mathrm{Liz},\Op T}(\R^d)$ can be identified as the continuous function
	\begin{equation}
		f(\V x) = \int_{\R^d}  \rho_{\Op T,\V y}(\V x) - p_{\V y}(\V x) \dint \mu(\V y)
	\end{equation}
	with bounded growth so that
	$\vert f(\V x) \vert \leq \vert \mu\vert (\R^d) g(\Vert \V x \Vert)$.
	In particular, we get that the operator $\Op P_{\mathrm{Liz},\Op T} \colon \Spc S^\prime_{\mathrm{Liz},\Op T}(\R^d) \to L_{\infty,\alpha}(\R^d) \hookrightarrow \Spc S^\prime (\R^d)$ with $\Op T^{-*}  \{\mu\} \mapsto f$ assigning the representatives is linear and continuous.
	Moreover, if $h(\V x, \cdot) \in C_0(\R^d)$ for every $\V x \in \R^d$, then the point evaluations for the representatives are in the predual $\Spc S_{\mathrm{Liz},\Op T}(\R^d)$.
\end{theorem}
\begin{proof}
	We have already shown that $f$ is indeed a representative. 
	The growth bound and the continuity of $\Op P_{\mathrm{Liz},\Op T}$ follow immediately from the bound on $h(\V x, \V y)$.
	
	Next, we show that point evaluations are weak*-continuous.
	Let $f_n = \Op T^{-*} \{\mu_n\}$ be a weak*-convergent sequence with limit  $ f = \Op T^{-*} \{\mu\}$, in the sense that $\mu_n$ converges weakly to $\mu$ as a measure.
	Due to the requirement that $e(\V x, \cdot) \in C_0(\R^d)$, we directly get that the evaluation functionals are weak*-continuous.
	To conclude the argument, we recall that the only weak*-continuous linear functionals on $\Spc S^\prime_{\mathrm{Liz},\Op T}(\R^d)$ are the elements of $\Spc S_{\mathrm{Liz},\Op T}(\R^d)$, see \cite[Thm.\ IV.20]{ReedSimon1980}.
\end{proof}
While this construction does not cover all Lizorkin distributions, we discuss two interesting examples, which will then be used in Section~\ref{sec:VarFramework} to revisit representer theorems for certain problems.
\begin{remark}\label{rem:RadialLiz}
	The same argumentations and constructions can be applied if $\Spc S_1$ consists of even or odd (hyper-spherical) Lizorkin functions.
	This setting is actually required for one of our examples.
\end{remark}

\subsection{Example 1: Fractional Laplacians}\label{sec:Liz_Frac}
Here, we choose the spaces as $\Spc S_1= \Spc S_2=\Spc S_\mathrm{Liz}(\R^d)$ and $\Spc X = C_0(\R^d)$ with $\Op T$ being the fractional Laplacian described in Section~\ref{Sec:Notations}, which is self-adjoint.
Specifically,
$(-\Delta)^{\alpha}\colon \Spc S_\mathrm{Liz}(\R^d) \to \Spc S_\mathrm{Liz}(\R^d)$ for any $\alpha \in \R$ with $(-\Delta)^{-\alpha}(-\Delta)^{\alpha}=\Identity$ on $\Spc S_\mathrm{Liz}(\R^d)$.
First, we note that the required density result was already established in Theorem~\ref{lem:DensLiz}.
According to these choices, $\Spc X^\prime_{\Op T}$ is given by  
\begin{align}
	\Spc M^{\alpha}(\R^d)= \bigl\{(-\Delta)^{-\alpha/2} \mu \in \Spc S'_\mathrm{Liz}(\R^d):\mu \in \Spc M(\R^d)\bigr\},\label{eq:RepMeas}
\end{align}
with predual space $C^\alpha(\R^d) = \overline{(\Spc S_{\mathrm{Liz}}(\R^d),\|(-\Delta)^{-\alpha/2}\cdot\|_{L_\infty})}$.

By Theorem~\ref{thm:RepLiz}, we can get a representation operator for $\alpha>d$ and $(\alpha -d)\notin \N$.
Indeed, let $\rho_{\mathrm{Liz},\alpha} = (-\Delta)^{-\alpha/2}\{\delta\}\in \Spc M^{\alpha}(\R^d)$ with a continuous representation given by \eqref{Eq:GreenLap}.
Using this representation, we obtain that $\rho_{\mathrm{Liz},\alpha} (\cdot- \V x_k) = (-\Delta)^{-\alpha/2}\{\delta(\cdot- \V x_k)\}$.
Now, we have to show that there exist polynomials $p_{\V y}(\V x) \in \Spc P(\R^d)$ such that the kernel $h(\V x, \V y) = \rho_{\mathrm{Liz},\alpha}(\V x - \V y)  - p_{\V y}(\V x)$  fulfills the requirements.
Based on \eqref{Eq:GreenLap}, we construct  for $\V y \neq \mathbf 0$ the polynomial $\tilde p_{\V y}(\V x)= T_{\lceil \alpha -d -1 \rceil} \{\rho_{\mathrm{Liz},\alpha}(\cdot - \V y)\}(\V x)$ with $T_{\lceil \alpha -d -1 \rceil}$ the Maclaurin expansion of order $\lceil \alpha -d -1 \rceil$ around $\V 0$.
For this function, we can bound the kernel $\tilde h(\V x, \V y) = \rho_{\mathrm{Liz},\alpha}(\V x - \V y)  - \tilde p_{\V y}(\V x)$ by
\begin{equation}
	\vert \tilde h(\V x, \V y)\vert \leq C \Vert \V x \Vert^{\lceil \alpha -d \rceil} \sup_{t \in [0,1], \vert \V k \vert = \lceil \alpha -d \rceil} \Vert \partial^{\V k} \rho_{\mathrm{Liz},\alpha}(t\V x - \V y) \Vert_2.
\end{equation}
For any fixed $\V x \in \R^d$, we then use our estimates from Proposition~\ref{eq:FundGrowth} in the appendix to conclude that $\vert \tilde h(\V x, \V y)\vert \leq C \Vert \V x \Vert^{\alpha -d}$ if $\Vert \V y \Vert \geq \Vert \V x \Vert + 1$ and $\tilde e(\V x,\cdot) \in C_0(\R^d)$ for every $\V x \in \R^d$.
Next, apply a smooth function $\chi\colon \R \to \R$ with $\chi(t)=0$ if $\vert t \vert \leq 1$ and $\chi(t)=1$ if $\vert t \vert \geq 2$ to define a bi-continuous (both in $\V x$ and $\V y$) function $x \mapsto p_{\V y}(\V x) = \chi(\Vert \V y \Vert) \tilde p_{\V y}(\V x) \in \Spc P (\R^d)$.
Now, we can bound the kernel $h(\V x, \V y)$ using the bound for $\tilde h(\V x, \V y)$ by
\begin{align}
	\vert h(\V x, \V y)\vert &\leq \max\Bigl\{ \max_{\Vert \V y \Vert \leq \Vert \V x \Vert + 2} \bigl\{\rho_{\mathrm{Liz},\alpha}(\V x - \V y) + \vert p_{\V y}(\V x) \vert\bigr\}, C \Vert \V x \Vert^{\alpha -d} \Bigr\} \leq \!C(\Vert \V x \Vert + 2)^{\alpha - d},
\end{align}
where we used Proposition~\ref{eq:FundGrowth} to produce the estimate
\begin{equation}
	\max_{\Vert \V y \Vert \leq \Vert \V x \Vert + 2} \vert p_{\V y}(\V x) \vert \leq \max_{\Vert \V y \Vert \leq \Vert \V x \Vert + 2} \sum_{\vert \V k \vert \leq \lceil \alpha -d -1 \rceil} \frac{1}{\V k!} \vert \partial^{\V k} \rho_{\mathrm{Liz},\alpha}(-\V y) \vert \Vert \V x\Vert ^{\vert \V k \vert} \leq \! C (\Vert \V x \Vert + 2)^{\alpha -d}.
\end{equation}
Finally, we  apply Theorem~\ref{thm:RepLiz} to obtain the desired representations.
The results of this section are summarized in the following corollary.
\begin{corollary}\label{cor:OperatorFrac}
	Let $\alpha>d$ with $(\alpha -d)\notin \N$.
	Then, there exists a continuous representation operator $\Op P_{\mathrm{Liz},\alpha} \colon \Spc M^{\alpha}(\R^d) \to L_{\infty,\alpha -d}(\R^d) \subset \Spc S'(\R^d)$ with $(-\Delta)^{-\alpha/2}  \{\mu\} \mapsto f$.
	Further, the point evaluations for these representatives are in the predual $C^{\alpha}(\R^d)$.
\end{corollary}

\subsection{Example 2: Radon Domain Splines }\label{sec:Liz_Radon}
In this example, we work with certain hyper-spherical counterparts of $\Spc S_\mathrm{Liz}(\R^d)$ as described in Section~\ref{Sec:Notations}.
More specifically, the Euclidean indexing with $\V x \in \R^d$ is replaced by $(t, \boldsymbol \xi) \in \R \times  \mathbb{S}^{d-1}$ and we distinguish between even and odd functions.
We express this distinction with an index $m \in \N$, which simplifies the notation. 
We define
\begin{equation}
	\Spc S_{\mathrm{Liz},m}(\R \times \mathbb S^{d-1}) = \Bigl\{\varphi \in \Spc S_{m}(\R \times \mathbb{S}^{d-1}): \int_{\R \times \mathbb{S}^{d-1}} \varphi(t,\boldsymbol \xi)p(t) \dint t \dint \boldsymbol \xi = 0 \,\,\, \forall p \in \Spc P(\R) \Bigr\},
\end{equation}
where $\dint \boldsymbol \xi$ stands for the surface element on the unit sphere $\mathbb{S}^{d-1}$.
Here, the space $\Spc S_{m}(\R \times \mathbb{S}^{d-1})$ is defined as the even functions in $\Spc S(\R \times \mathbb{S}^{d-1})$ if $m$ is even and the odd ones otherwise.
Correspondingly, an element \smash{$g \in \Spc S_{\mathrm{Liz},m}'(\R \times \mathbb{S}^{d-1})$} is a continuous linear functional on $\Spc S_{\mathrm{Liz},m}(\R \times \mathbb{S}^{d-1})$ whose action on the test function $\phi$ is represented by the duality product $\langle g,\phi \rangle_\mathrm{Rad}$.
If $g$ can be identified with a function $g\colon \R \times \mathbb{S}^{d-1} \to \R$, then
\begin{align}
	\langle g,\phi\rangle_\mathrm{Rad} = \int_{\mathbb{S}^{d-1}} \int_{\R} g(t, \boldsymbol \xi) \phi(t, \boldsymbol \xi) \dint t \dint \boldsymbol \xi.
\end{align}
The evaluation functional on $\Spc S_{\mathrm{Liz}}(\R \times \mathbb S^{d-1})$ is $\delta_{\V z_0} = \delta(\cdot-t_0) \delta(\cdot - \boldsymbol \xi_0) $ with $\V z_0 = (t_0,\boldsymbol \xi_0) \in \R \times \mathbb S^{d-1}$.
A brief overview for properties of the Radon transform $\Op R$ and its filtered version $\Op K_\mathrm{rad}\Op R$ related to these spaces is given in Appendix~\ref{sec:Radon}.
In particular, it holds that both $\Op R$ and $\Op K_\mathrm{rad}\Op R$ are homeomorphisms.
Next, we briefly review Lizorkin ridges, which play a key role for the construction of representatives.

\paragraph{Lizorkin Ridges}
The 1D profile (or ridge) along the direction $\boldsymbol \xi_0 \in \R^d$ associated to $r \in \Spc S_{\mathrm{Liz}}'(\R)$ is the distribution $r_{\boldsymbol \xi_0 } \in \Spc S_\mathrm{Liz}'(\R^d)$
that satisfies
\begin{align}
	\label{Eq:Ridges}
	\forall \varphi \in \Spc S_\mathrm{Liz}(\R^d): \quad \langle r_{\boldsymbol \xi_0}, \varphi\rangle=\langle r, \Op R\{\varphi\} (\cdot,\boldsymbol \xi_0)\rangle.
\end{align}
The most basic ridge is $\delta(\boldsymbol \xi_0^\Top \cdot-t_0) \coloneqq r_{\boldsymbol \xi_0}$ with $r=\delta(\cdot-t_0)$.
It is a Dirac ridge along $\boldsymbol \xi_0$  with offset $t_0$.
Since the Fourier transform of such ridges is localized along the ray $\{\bw=\omega\boldsymbol \xi_0: \omega \in\R\}$, the Radon transform of a ridge must vanish  away from $\pm \boldsymbol \xi_0$.
This is generalized and formalized as follows.
\begin{proposition}[Radon transform of Lizorkin ridges]
	\label{Theo:Rad1DprofilesNew}
	Let $(t_0,\boldsymbol \xi_0)=\V z_0 \in \R \times \mathbb{S}^{d-1}$ and $r \in \Spc S_\mathrm{Liz}'(\R)$.
	Then,
	\begin{align}
		\Op K_\mathrm{rad}\Op R \{ \delta(\boldsymbol \xi_0^\Top\cdot- t_0) \}&=\Op P_\mathrm{even}\{\delta_{\V z_0}\}\in \Spc S_\mathrm{Liz}'(\R \times \mathbb{S}^{d-1})\label{eq:1Id}\\
		\Op R\{ \delta(\boldsymbol \xi_0^\Top\cdot) \}
		&=
		\Op P_\mathrm{even}\{ q_d\delta(\cdot -\boldsymbol \xi_0)\}\in \Spc S_\mathrm{Liz}'(\R \times \mathbb{S}^{d-1})\label{eq:2Id}\\
		\Op K_\mathrm{rad}\Op R\{ r(\boldsymbol \xi_0^\Top\cdot) \}
		&=
		\Op P_\mathrm{even}\{ r\delta(\cdot-\boldsymbol \xi_0)\}\in \Spc S_\mathrm{Liz}'(\R \times \mathbb{S}^{d-1})\label{eq:3Id}\\
		\Op R \{ r(\boldsymbol \xi_0^\Top\cdot) \}
		&=\Op P_\mathrm{even}\{ (q_d \ast r)\delta(\cdot-\boldsymbol \xi_0)\} \in \Spc S_\mathrm{Liz}'(\R \times \mathbb{S}^{d-1}),\label{eq:4Id}
	\end{align}
	where $q_d(t)=2(2\pi)^{d-1}\Fourier^{-1}\{ 1 / |\cdot|^{d-1}\}(t)$ is the 1D impulse response of the Radon-domain inverse filtering operator $\Op K_\mathrm{rad}^{-1}$.
	Here, \eqref{eq:1Id} can be identified as an even measure.
\end{proposition}

\begin{proof}
	For any $\varphi \in \Spc S_{\mathrm{Liz},0}( \R \times \mathbb{S}^{d-1})$, it holds that $\Op R \Op R^* \Op K_\mathrm{rad} \{\varphi\} = \varphi$ and, therefore, also that
	\begin{align}
		\langle \Op K_\mathrm{rad}\Op R \{r_{\boldsymbol \xi_0}\}, \varphi \rangle= \langle r_{\boldsymbol \xi_0}, \Op R^* \Op K_\mathrm{rad} \{\varphi\} \rangle = \langle r, \Op R \Op R^* \Op K_\mathrm{rad} \{\varphi\} (\cdot,\boldsymbol \xi_0)\rangle = \langle r, \varphi (\cdot,\boldsymbol \xi_0)\rangle,
	\end{align}
	from which \eqref{eq:1Id} and  \eqref{eq:3Id} do follow.
	In a similar way, we obtain, for any $\varphi \in \Spc S_{\mathrm{Liz},0}( \R \times \mathbb{S}^{d-1})$, that
	\begin{align}
		\langle \Op R \{r_{\boldsymbol \xi_0}\}, \varphi \rangle= \langle r, \Op R \Op R^* \{\varphi\} (\cdot,\boldsymbol \xi_0)\rangle = \langle r, \Op K_\mathrm{rad}^{-1} \{\varphi\} (\cdot,\boldsymbol \xi_0)\rangle,
	\end{align}
	from which \eqref{eq:2Id} and  \eqref{eq:4Id}  do follow as $\Op K_\mathrm{rad}^{-1} \{\varphi\} (t, \boldsymbol \xi_0) = (q_d \ast \varphi(\cdot, \boldsymbol \xi_0))(t)$.
\end{proof}

An equivalent form of \eqref{eq:1Id} in Proposition~\ref{Theo:Rad1DprofilesNew} is
\begin{align}
	\delta(\boldsymbol \xi_0^\Top\cdot- t_0) &=\Op R^\ast \Op P_\mathrm{even}\{\delta_{\V z_0}\}(\V x),
\end{align}
which results from $\Op R^\ast\Op K_\mathrm{rad}\Op R=\Identity$ on $\Spc S_\mathrm{Liz}'(\R^d)$.
Note that the other identities can be rewritten in a similar form, too.

\paragraph{Construction of Radon Splines}
In this example, we choose the spaces for constructing the Banach subspaces as $\Spc S_1= \Spc S_{\mathrm{Liz}, m}(\R \times   \mathbb{S}^{d-1})$, $\Spc S_2=\Spc S_\mathrm{Liz}(\R^d)$, and $\Spc X = C_{0,m}(\R^d)$, where $C_{0,m}(\R^d)$ consists of even or odd continuous functions, respectively, that vanish at infinity.
Next, recall that the derivative $\partial_t^m \colon \Spc S_{\mathrm{Liz}, m}(\R \times   \mathbb{S}^{d-1}) \to  \Spc S_\mathrm{Liz,0}(\R \times   \mathbb{S}^{d-1})$ is self-adjoint and a homeomorphism.
Its inverse can be constructed by iterating $\partial_t^{-1}\{\varphi\}(t,\boldsymbol \xi) = \int_t^\infty \varphi(r,\boldsymbol \xi) \dint r$.
Then, we choose $\Op T = \Op R^* \Op K_\mathrm{rad} \partial_t^m$
such that the dual $\Op T^*= \partial_t^m \Op K_\mathrm{rad}\Op R$ is  the concatenation of the filtered projection $\Op K_\mathrm{rad}\Op R\colon \Spc S_\mathrm{Liz}'(\R^d) \to \Spc S'_\mathrm{Liz,0}(\R \times   \mathbb{S}^{d-1})$ with the partial derivative $\partial_t^m$.
To begin, we show the required density result for the construction of the spaces related to Theorem~\ref{thm:RepLiz}, which also applies in this hyper-spherical setting, as pointed out in Remark~\ref{rem:RadialLiz}.
\begin{lemma}
	It holds that
	\begin{equation}
		\overline{\bigl(\Spc S_{\mathrm{Liz},m}(\R^d), \Vert \cdot \Vert_\infty\bigr)} = C_{0,m}(\R^d) = \begin{cases}
		C_{0,\mathrm{even}} (\R^d) &\mbox{ if $m$ is even}\\
		C_{0,\mathrm{odd}}(\R^d) &\mbox{ if $m$ is odd}.
	\end{cases}
	\end{equation}
\end{lemma}
\begin{proof}
	By the Stone--Weierstrass theorem and the continuity of the projection onto even or odd functions, respectively, we first get that
	$\overline{C_{0,m}(\R) \times C(\mathbb S^{d-1})}= C_{0,m}(\R^d)$.
	Then, we conclude from Theorem~\ref{lem:DensLiz} that $\Spc S_\mathrm{Liz, m}(\R) \times  C^\infty (\mathbb{S}^{d-1}) \subset \Spc S_\mathrm{Liz, m}(\R \times   \mathbb{S}^{d-1})$ is dense in $C_{0,m}(\R^d)$.
\end{proof}
According to these choices, our Banach space $\Spc X^\prime_{\Op T}$ with smoothness exponent $m$ is given by  
\begin{align}
	\Spc M_{\mathrm{Rad},m}(\R^d)=\bigl\{ \Op R^* \partial_t^{-m} \{\mu\} \in \Spc S_{\mathrm{Liz}}'(\R^d) :\mu \in \Spc M_{m}(\R^d)\bigr\}\label{eq:RepMeas2}
\end{align}
with predual \smash{$C_{\mathrm{Rad},m}(\R^d) = \overline{(\Spc S_\mathrm{Liz}(\R^d),\|\partial_t^{-m} \Op R \{\cdot\}\|_{L_\infty})}$}.
Now, we show that Theorem~\ref{thm:RepLiz} can be applied for $m \geq 2$ to get continuous representations of elements in \smash{$\Spc M_{\mathrm{Rad},m}(\R^d)$}.
Define
\begin{equation}\label{eq:RadonBasis}
	\rho_{\mathrm{Rad}, m} (x) = \max(0, x )^{m-1}/(m-1)!.
\end{equation}
Then, as shown in Proposition~\ref{Theo:Rad1DprofilesNew}, $\rho_{\mathrm{Rad}, m}(\langle \cdot, \boldsymbol \xi_0 \rangle - t_0)$  with $\V z_0 = (t_0, \boldsymbol \xi_0) \in \R \times \mathbb S^{d-1} $  is an element of the equivalence class
\begin{equation}
	\tfrac12 \Op R^*  \Op P_\mathrm{even} \{\rho_{\mathrm{Rad}, m}(\cdot-t_0)\delta(\cdot-\boldsymbol \xi_0)\}= \tfrac12 \Op R^* \partial_t^{-m} \{\delta_{\V z_0} \pm \delta_{\V -\V z_0}\} \in \Spc M_{\mathrm{Rad}, m}(\R^d),
\end{equation}
where the sign depends on $m$.
Now, we have to show that there are polynomials $p_{t,\boldsymbol \xi} \in \Spc P(\R^d)$ such that the kernel $h(\V x, \V z) = \rho_{\mathrm{Rad}, m}(\langle \V x, \boldsymbol \xi \rangle - t)  - p_{t,\boldsymbol \xi}(\V x)$  with $z=(t,\boldsymbol  \xi)$ fulfills the requirements.
As $m$ is a natural number, no Taylor expansion is necessary and we can provide the correcting family of polynomials directly.
More precisely, we set
\begin{equation}
	p_{t,\boldsymbol  \xi} = \max\bigl\{0,\min\{-t,1\}\bigr\} \bigl(\langle \cdot, \boldsymbol  \xi \rangle - t\bigr)^{m-1}/(m-1)! \in \Spc P(\R^d),
\end{equation}
which ensures that $(\rho_{\mathrm{Rad},m}(\langle \V x, \cdot \rangle - \cdot) - p_{\{\cdot\}}(\V x)) \in C_{0}(\R \times   \mathbb{S}^{d-1})$ together with
\begin{equation}
	\Vert \rho_{\mathrm{Rad},m}(\langle \V x, \cdot \rangle - \cdot) - p_{\{\cdot\}}(\V x)\Vert_\infty \leq C \Vert \V x \Vert^{m-1}.
\end{equation}
Hence, we can apply Theorem~\ref{thm:RepLiz} to obtain explicit  representations.

\begin{corollary}\label{cor:OperatorRadon}
	Let $m \geq 2$.
	Then, there exists a continuous representation operator $\Op P_{\mathrm{Rad},m } \colon \Spc M_{\mathrm{Rad},m}(\R^d) \to L_{\infty,m-1}(\R^d)\subset \Spc S'(\R^d)$ with $\Op R^* \partial_t^{-m} \{\mu\} \mapsto f$.
	Further, the point evaluations for these representatives are in the predual $C_{\mathrm{Rad},m}(\R^d)$.
\end{corollary}

\section{Variational Problems that Involve Lizorkin Spaces}\label{sec:VariationalProb}
As warm-up, we first revisit periodic (fractional) splines \cite{Fageot2020tv}.
For variational problems in which the regularization favors such functions, we can use the projection \eqref{Eq:ProjPerio} to get representations.
Note that our approach is applicable to a very broad class of problems; namely, whenever a continuous projection and a suitable extension are available.
By contrast, after this warm-up example, we study variational problems where no projection onto the involved spaces is available.
There, we focus on problems that involve our previously constructed Banach subspaces, which usually only consist of equivalence classes.
This  makes the situation much more delicate than before and the use of a representation operator is necessary.
Based on this operator, we are able to obtain similar results as before.

\subsection{Periodic Fractional Splines}\label{sec:PeriodicLizSpline}
We use our tools to derive a representer theorem that is an alternative to the one in  \cite{Fageot2020tv}.
To this end, we need the space
$C(\mathbb{T})=\overline{(\Spc S(\mathbb{T}),\|\cdot\|_{L_\infty})}$ of continuous, $T$-periodic functions.
Its topological dual $\Spc M(\mathbb{T})$ (namely, the space of $T$-periodic Radon measures) can be specified as
\begin{align}
	\Spc M(\mathbb{T})=\bigl\{f\in \Spc S'(\mathbb{T}): \|f\|_{\Spc M}<\infty\bigr\}
	\quad \text{ with } \quad
	\|f\|_{\Spc M} \coloneqq \sup_{\phi \in \Spc S(\mathbb{T}): \|\phi\|_{L_\infty}\le 1}
	\langle f,\phi\rangle.
\end{align}
Since the projection \eqref{Eq:ProjPerio}  continuously extends to these spaces, we have the decomposition $C(\mathbb{T})=C_0(\mathbb{T}) \oplus \Spc P_0$  with $C_0(\mathbb{T})=\Op P_0(C(\mathbb{T}))$.
The final ingredient are the sampling functionals in $\Spc M_0(\mathbb{T}) = C_0(\mathbb{T})' \simeq \Spc M(\mathbb{T}) / \Spc P_0$, where we use $\Op P_0^*$ to identify representations.
\begin{theorem}[Periodic Lizorkin sampling functionals]
	\label{Theo:ExtremePerio}
	The Lizorkin sampling functionals $\delta_{0}(\cdot-t_0)=\Op P_0^* \{\delta_\mathrm{perio}(\cdot - t_0)\} \in \Spc M_0(\mathbb{T})$ with $t_0 \in \mathbb{T}$ have the following properties:
	
	\begin{enumerate}
		\item Explicit representation: $\displaystyle \delta_0(\cdot-t_0)=\delta_\mathrm{perio}(\cdot - t_0)-1$.
		\item Sampling at $t_0$: $
		\langle  \delta_{0}(\cdot-t_0),\phi\rangle=\phi(t_0)$ for all $\phi \in C_{0}(\mathbb{T}).$
		\item Zero mean: $\langle \delta_{0}(\cdot-t_0),1\rangle=0$ for all $t_0\in \R$.
		\item  
		It holds that  $\|\delta_{0}(\cdot-t_0)\|_{\Spc M_0}=1$ for any $t_0 \in \mathbb{T}$.
		\item For finite sets $\{t_k\}\subset\mathbb{T}$ of distinct points, it holds $\|\sum_{k} a_k\delta_0(\cdot-t_k)\|_{\Spc M_0}=\sum_{k} |a_k|$.
		\item If $e_k\in \mathrm{Ext}B(\Spc M_0(\mathbb{T}))$, then $e_k=\pm\delta_0(\cdot-t_k)$ for some $t_k \in\mathbb{T}$.
	\end{enumerate}
	
\end{theorem}

\begin{proof} The first 3 items follow directly by construction.
	Now, we prove Item 4.
	For any $(f,\phi) \in \Spc M(\mathbb{T}) \times C_0(\mathbb{T})$, it holds that
	\begin{align}
		\langle\Op P^\ast_0\{ f\},  \phi \rangle=\langle f, \Op P_0 \{\phi\} \rangle=\langle f, \phi \rangle.
	\end{align}
	In particular, $\langle \delta_{0}(\cdot-t_0),\phi\rangle=\langle \delta_\mathrm{perio}(\cdot-t_0),\phi\rangle$.
	By definition of the dual norm, we have that
	\begin{align}
		\|\delta_0(\cdot-t_0)\|_{\Spc M_0}&=\sup_{\phi \in C_0(\mathbb{T}): \|\phi\|_{L_\infty}\le 1}
		\langle \delta_0(\cdot-t_0),\phi\rangle =\sup_{\phi \in C_0(\mathbb{T}): \|\phi\|_{L_\infty}\le 1}
		\langle \delta_\mathrm{perio}(\cdot-t_0),\phi\rangle\notag\\
		&\le\sup_{\phi \in C(\mathbb{T}): \|\phi\|_{L_\infty}\le 1}
		\langle \delta_\mathrm{perio}(\cdot-t_0),\phi\rangle=\|\delta_\mathrm{perio}(\cdot-t_0)\|_{\Spc M}=1.
	\end{align}
	Next, we show that this bound is sharp by fixing $0<\epsilon<T/2$ 
	and choosing the test function
	\begin{equation}
		\phi_{\epsilon,\mathrm{perio}}(\cdot - t_0)=\sum_{n \in \Z} \varphi_0\Big( \frac{ \cdot + n T -t_0}{\epsilon}\Big),
	\end{equation}
	where $\varphi_0\colon \R \to [-1,1]$ is continuous with $\varphi_0(0)=1$, $\int_\R \varphi_0(t)\dint t=0$, and $\mathrm{supp}(\varphi_0)\subset [-1,1]$.
	Then, the statement follows from $\phi_{\epsilon,\mathrm{perio}}(t_0)=1$
	and $\|\phi_{\epsilon,\mathrm{perio}}\|_{L_\infty}\le 1$.
	Similarly, for Item 5, we first observe that the triangle inequality leads to
	\begin{equation}
		\sup_{\phi \in C_0(\mathbb{T}): \|\phi\|_{L_\infty}\le 1}
	\Bigl \langle\sum_{k} a_k\delta_0(\cdot-t_k),\phi\Bigr \rangle=\Bigl\|\sum_{k} a_k\delta_0(\cdot-t_k)\Bigr\|_{\Spc M_0}\le \sum_{k} |a_k|.
	\end{equation}
	Since the $t_k$ are distinct, there exists $\epsilon>0$ with
	$|t_k-t_{k'}|> 2 \epsilon$ for all $k'\ne k$.
	Then, we take the critical function
	$\phi_\mathrm{crit}(t)=\sum_k \mathrm{sgn}(a_k)\phi_{\epsilon,\mathrm{perio}}(t-t_k)$ satisfying $\|\phi_\mathrm{crit}\|_{\mathrm L_\infty}=1$,
	which saturates the bound.
	
	Due to $M_0(\mathbb{T}) \simeq \Spc M(\mathbb{T}) / \Spc P_0$,  it holds that $ \Op P_0^* B(\Spc M(\mathbb{T})) =  B(\Spc M_0(\mathbb{T}))$.
	By \cite[Lem.~3.2]{Bredies2020}, we then get that $\mathrm{Ext}B(\Spc M_0(\mathbb{T})) \subset \Op P_0^\ast \mathrm{Ext} B(\Spc M(\mathbb{T}))$.
	Since the extreme points of $B(\Spc M(\mathbb{T}))$ are $\{\pm \delta(\cdot-t)\}_{t \in \mathbb{T}}$, the last claim readily follows.
\end{proof}

Now, we are  able to formulate the approximation problem.
Given a series of (possibly noisy) data points $(y_m,t_m)\in \R \times \mathbb{T}$, $m=1, \dots, M$, we consider the task of reconstructing a periodic function
$f\colon \mathbb{T}\to \R$ such that $f(t_1)\approx y_1, \dots, f(t_M)\approx y_M$ without overfitting.
Since this problem is inherently ill-posed, we put a penalty on $\Vert \Op D^\alpha \{f\} \Vert_{\Spc M_0}$ in order to favor solutions with ``sparse'' $\alpha$th derivatives. The corresponding native space is 
\begin{align}
	\Spc M^\alpha(\mathbb{T})&=\{f \in \Spc S'(\mathbb{T}): \|\Op D^\alpha \{f\}\|_{\Spc M_0}<\infty\}\nonumber \\
	&=\{\Op D^{-\alpha}\{w\} + p_0: (w,p_0) \in \Spc M_0(\mathbb{T}) \times \Spc P_0\}.
\end{align}
In particular, this means that $\Spc M^\alpha(\mathbb{T})=\Spc U' \oplus \Spc P_0$ with $\Spc U'=\Op D^{-\alpha}(\Spc M_0(\mathbb{T}))$, which is isomorphic to  $\Spc M_0(\mathbb{T}) \times \Spc P_0$.
The basic atoms for the  representation of minimum-norm interpolators in $\Spc U'$ are the extreme points $e_k$ of the unit ball $B_{\Spc U'}(1)$.
Due to the isometry between $\Spc U'$ and  $\Spc M_0(\mathbb{T})$,
we have that $\mathrm{Ext}B_{\Spc U'}(1)=\Op D^{-\alpha}(\mathrm{Ext}B_{\Spc M_0}(1))$, which in light of Items 1 and 6 in Theorem~\ref{Theo:ExtremePerio} yields that
\begin{align}
	e_k=\Op D^{-\alpha}\{\delta_0(\cdot-t_k)\}=\rho_{\mathrm{perio}, \alpha}(\cdot -t_k),
\end{align}
where
\begin{align}
	\label{Eq:PerioGreen}
	\rho_{\mathrm{perio}, \alpha}(t)=\Op D^{-\alpha}\{\delta_0\}(t)=\sum_{n \in \Z \backslash \{0\}} \frac{1}{(\jj n \omega_0)^\alpha} \ee^{\jj n \omega_0 t}.
\end{align}
The latter formula is obtained from \eqref{Eq:FracInt} by using that $\widehat \delta_0[n] =\widehat  \delta[n]$ for $n \neq 0$.
The resulting Fourier series \eqref{Eq:PerioGreen} converges to a continuous function for $\alpha>1$.
The functions $\rho_{\mathrm{perio}, \alpha}$ are the building blocks of the (non-periodic) fractional splines of degree $\alpha-1$.
Now, the direct application of the third case of  \cite[Thm.\ 3]{Unser2022} yields the following.
\begin{theorem} [Minimum-energy periodic spline reconstruction]\label{thm:PerSPline}
	Let $E\colon \R \times \R \to \R$ be a strictly convex loss function and $\lambda>0$ some regularization parameter.
	Then, for any given data points $(y_m,t_m)\in \R \times \mathbb{T}, m=1, \dots, M$, the solution set of the
	functional-approximation problem with $\alpha> 1$,
	\begin{align}\label{eq:SplineMinimization}
		S=\argmin_{f \in \Spc M^\alpha(\mathbb{T}) } \sum_{m=1}^M E\big(y_m ,f(t_m)\big) +  
		\lambda \|\Op D^\alpha f\|_{\Spc M_0},
	\end{align}
	is nonempty and weak*-compact. It is the weak* closure of the convex hull of its extreme points, which are all of the form
	\begin{align}
		f_\mathrm{Ext}(t)= b_0+\sum_{k=1}^{K_0} a_k 
		\rho_{\mathrm{perio},\alpha}(t- \tau_k)
		\label{Eq:Extremesplineperio}
	\end{align}
	for some $K_0\le M-1$, weights and knots $(a_k,\tau_k) \in \R \times \R$, $k=1,\dots,K_0$, and the periodic basis function $\rho_{\mathrm{perio},\alpha}\colon \R \to \R$  specified by \eqref{Eq:PerioGreen}. 
\end{theorem}
\begin{proof}
	First, we identify the (unique) predual space $C_0^\alpha(\mathbb{T})=\Spc U \oplus \Spc P_0$
	such that $\Spc M^\alpha(\mathbb{T})=\Spc U' \oplus \Spc P'_0$.
	By the injectivity of $\Dop^{\alpha\ast}$ on $C_0(\mathbb{T})=\overline{(\Spc S_0(\R),\|\cdot\|_{L_\infty})}$ and by setting $\Spc U=\Op D^{\alpha\ast}(C_0(\mathbb{T}))$, we readily verify that
	$\Spc U'=\Op D^{-\alpha}(\Spc M_0(\mathbb{T}))$. This allows us to identify the predual space as
	\begin{align}
		C_0^\alpha(\mathbb{T})&=\Spc U \oplus \Spc P_0=\bigl\{\Dop^{\alpha\ast}\{v\} + p_0: (v,p_0) \in C_0(\mathbb{T})\times \Spc P_0\bigr\},
	\end{align}
	which is a Banach space isomorphic to $C_0(\mathbb{T})\times \Spc P_0$ as expected.
	The technical prerequisite for applying \cite[Thm.\ 3]{Unser2022} is the weak*-continuity of the sampling functionals $\delta(\cdot-t_m)$, which is equivalent to $\delta(\cdot-t_m) \in \Spc C_0^\alpha(\mathbb{T})$.
	To this end, we have that
	\begin{align}
		\Dop^{-\alpha\ast}\{\delta(\cdot-t_m)-1\}=\rho_{\mathrm{perio}, \alpha}(t_m-\cdot),
	\end{align}
	with the latter function being included in $C_0(\mathbb{T})$ if and only if $\alpha>1$ or, equivalently, when the Fourier coefficients in \eqref{Eq:PerioGreen} are in $\ell_1(\Z)$.
\end{proof}
\begin{remark}
	 Functions of the form \eqref{Eq:Extremesplineperio} are fractional splines if and only if $\sum_{k=1}^{K_0} a_k = 0$, see  \cite[Prop.\ 3]{Fageot2020tv}.
	 To ensure this, we can add the constraint $\langle f, e^{\jj \omega_0 \cdot} \rangle_{\mathbb T}= 0$ in Theorem~\ref{thm:PerSPline}, which again leads to extreme points of the form \eqref{Eq:Extremesplineperio} with $K_0\le M$.
	 Since $\Vert \Op D^\alpha f \Vert_{\Spc M} \geq \Vert \Op D^\alpha f \Vert_{\Spc M_0}$ with equality holding for extreme points, this modified version of Theorem~\ref{thm:PerSPline} remains true if we replace $\Vert  \Op D^\alpha f \Vert_{\Spc M_0}$ in \eqref{eq:SplineMinimization} by $\Vert  \Op D^\alpha f \Vert_{\Spc M}$.
	 Plots of the fractional splines $\rho_{\mathrm{perio},\alpha} - \rho_{\mathrm{perio},\alpha}(\cdot - \frac{T}{2})$ for different $\alpha$ are given in \cite[Figure 1]{Fageot2020tv}.
\end{remark}
\begin{remark}[Numerical approach]
	To find $f$, we can overparameterize it with knots $\tau_k$ chosen over a fine uniform grid.
	The respective weights are then recovered by solving a discrete penalized basis pursuit problem using state-of-the-art proximal algorithms \cite{DeFaGu2019,GuFaUn2018} or Bregman methods \cite{BuRoTe2022}.
	While conceptually simple, this is computationally expensive since the underlying grid needs to have much more knots than $M-1$. 
	More advanced meshfree approaches for directly recovering the positions $\tau_k$ can be developed using, for example, the Franck--Wolfe algorithm~\cite{DeDuPe2019,FlGoWe2020}.
\end{remark}

\subsection{A General Variational Problem Framework}\label{sec:VarFramework}
In this section, we first state a general variational problem framework that involves the constructed Banach spaces and shares some similarities with the approach presented in Section~\ref{sec:PeriodicLizSpline}, but for which no projector is available.
Here, the derived representation operator from Theorem~\ref{thm:RepLiz} makes the framework explicit, again with the advantage that we can rely on the general abstract machinery for the derivation of theoretical results.
We then treat several useful special cases related with the Banach subspaces of $\Spc S^\prime_\mathrm{Liz}(\R^d)$ introduced as examples in Section~\ref{sec:LizorkinRep}.

By construction, we immediately deduce that the extreme points of the unit ball in $\Spc X^\prime_{\Op T}$ are given by $\tilde e_k= \Op T^{-*}  \{e_k\}\in \Spc X^\prime_{\Op T}$, where $e_k$ are the extreme points of the unit ball in $\Spc X^\prime$.
Now, we are able to formulate a variational problem that involves our constructed Banach spaces  and provide a representer theorem for the structure of the solutions.
\begin{theorem}[Representer theorem \cite{Unser2022}]\label{thm:representer}
	Let  the linear operator $\nu \colon \Spc X_{\Op T}^\prime \to \R^M$ be given by $f \mapsto(\langle \nu_1,f\rangle,\ldots,\langle \nu_M,f\rangle)$ with $\nu_i \in \Spc X_{\Op T}$ being linearly
	independent.
	Further, let $E \colon  \R^M \times \R^M \to \R_+ \cup \{+ \infty\}$ be proper, lower-semicontinuous, and convex and let $\psi \colon \R_+ \to \R_+$ be strictly increasing and convex.
	Then, for any fixed $y \in \R^M$, the solution set $S$ of the generic optimization problem
	\begin{equation}
		\argmin_{f \in \Spc X^\prime_{\Op T}} E\bigl(y, \nu\{f\}\bigr) +  \psi (\Vert f \Vert_ {\Spc X^\prime_{\Op T}})
	\end{equation}
	is nonempty, convex, and weak*-compact.
	If, additionally, $E$ is strictly convex or if it imposes the equality constraint $y = \nu\{f\}$, then $S$ is the weak* closure of the convex hull of its extreme points, which can all be expressed as
	\begin{equation}
		f_0 = \sum_{k=1}^{K_0} c_k \Op T^{-*} \{e_k\}
	\end{equation}
	with $K_0 \leq M$ and $c_k \in \R$.
\end{theorem}
\begin{remark}
	The result can be slightly strengthened if $\Spc X^\prime$ is strictly convex, see \cite{Unser2022} for details.
\end{remark}
As illustration, we briefly derive two corollaries from Theorem~\ref{thm:representer}.
They are based on the two Banach subspaces of $\Spc S^\prime_\mathrm{Liz}(\R^d)$ introduced in Section~\ref{sec:LizorkinRep}.

\subsection{Fractional Splines}\label{sec:LizSplines}
Here, we extend our investigations in Section~\ref{sec:PeriodicLizSpline} to non-periodic splines using Theorem~\ref{thm:representer} and the discussion from Section~\ref{sec:Liz_Frac}, which is summarized in Corollary~\ref{cor:OperatorFrac}.
Since the point evaluations are in the predual, the application of Theorem~\ref{thm:representer}, together with the explicit representation of elements in $ \Spc M^\alpha(\R^d)$, yields the following.
\begin{corollary} [Minimum-energy Lizorkin splines]
	Let $E\colon \R \times \R \to \R$ be a strictly convex loss function,
	let $(\V x_m,y_m)\in \R^d \times \R, m=1,\dots,M$, be a set of data points, and let $\lambda>0$ be some regularization parameter.
	Then, for $\alpha>d$ and $\alpha -d \notin \N$, the solution set $S$ of the
	functional optimization problem  
	\begin{align}
		\argmin_{f \in \Spc M^\alpha(\R^d) }  \sum_{m=1}^M E\bigl(y_m ,P_{\mathrm{Liz},\alpha} \{f\}(\V x_m)\bigr) +  
		\lambda \|(-\Delta)^{\alpha/2} \{f\}\|_{\Spc M}
	\end{align}
	is nonempty and weak*-compact. It is the weak* closure of the convex hull of its extreme points, which are all of the form
	\begin{align}
		f_\mathrm{Ext}= \Op P_{\mathrm{Liz},\alpha}\biggl\{\sum_{k=1}^{K_0} a_k (-\Delta)^{\alpha/2} \{\delta(\cdot -\V x_k)\}\biggr\}= \sum_{k=1}^{K_0} a_k 
		\bigl(\rho_{{\mathrm{Liz},\alpha}}(\cdot -  \V x_k) - p_{\V x_k}\bigr)
		\label{Eq:ExtremesplineLiz}
	\end{align}
	for some $K_0\le M$, expansion parameters (weights and adaptive centers) $(a_k,\V x_k) \in \R \times \R^d$ for $k=1,\dots,K_0$, $p_{\V x_k} \in \Spc P_{\lceil \alpha -d - 1\rceil}(\R^d)$, and the radial basis function $\rho_{{\mathrm{Liz},\alpha}}\colon \R^d \to \R$ from Section~\ref{sec:Liz_Frac}. 
\end{corollary}

\begin{remark}
	Starting from the chosen representative, we could  replace $P_{\mathrm{Liz},\alpha} \{f\}(\V x_m)$ with $P_{\mathrm{Liz},\alpha} \{f\}(\V x_m) + p(\V x_m)$, where $p\in \Spc P_{\lfloor \alpha -d \rfloor}(\R^d)$, which would result in a minimization over $L_{\infty,\alpha -d}(\R^d)$.
	Hence, we are back to a more classical setting and a similar result holds, see \cite[Thm.\ 3]{Unser2022}.
	Proving the weak*-continuity of the evaluation functional in this extended setting follows along the lines of the Lizorkin-distribution setting.
\end{remark}

\subsection{Radon Splines}\label{sec:RadonReLU}
The native Banach space for interpolation with Radon splines, given some order $m\in \N$, is the space $ \Spc M_{\mathrm{Rad},m}(\R^d)$ introduced in Section~\ref{sec:Liz_Radon}.
Due to the form of the function $\rho_{\mathrm{Rad},m}$, this interpolation problem is closely related to approximations with 2-layer neural networks as pointed out in \cite{Bartolucci2021,Parhi2021}.
Since the point evaluations are in the predual, the application of Theorem~\ref{thm:representer}, together with the explicit representation of elements in $ \Spc M_{\mathrm{Rad},m}(\R^d)$ obtained in Corollary~\ref{cor:OperatorRadon}, yields the following.
	
\begin{theorem} [Minimum-energy Radon splines]
		\label{Theo:RadonSplines}
		Let $E\colon \R \times \R \to \R$ be a strictly convex loss function,
		let $(\V x_i,y_i)\in \R^d \times \R, i=1,\dots,M$, be a set of data points, and let $\lambda>0$ be some regularization parameter.
		For $m \in \N, m \geq 2$, the solution set $S$ of the
		functional optimization problem 
		\begin{align}
			\label{Eq:RadonEnergy}
			\argmin_{f \in \Spc M_{\mathrm{Rad},m}(\R^d) }  \sum_{i=1}^M E\bigl(y_i ,P_{\mathrm{Rad},m}\{f\}(\V x_i)\bigr) +  
			\lambda\|\partial_t^m \Op K_\mathrm{rad}\Op R\{f\}\|_{\Spc M_{m}},
		\end{align}
		is nonempty and weak*-compact. It is the weak* closure of the convex hull of its extreme points, which are all of the form
		\begin{align}
			f_\mathrm{Ext}&= \Op P_{\mathrm{Rad},m}\biggl\{\sum_{k=1}^{K_0} a_k \Op R^* \partial_t^{-m}  \{\delta(\cdot -\V x_k)\}\biggr\} = \sum_{k=1}^{K_0} a_k 
			\bigr(\rho_{\mathrm{Rad},m}(\langle \boldsymbol \xi_k,\cdot \rangle -t_k) - p_{t_k,\boldsymbol \xi_k}\bigr)
			\label{Eq:ExtremesplineLik}
		\end{align}
		for some $K_0\le M$, expansion parameters (weights and adaptive centers) $(a_k,t_k,\boldsymbol \xi_k) \in \R \times \R \times \mathbb S^{d-1}$ for $k=1,\dots,K_0$, $p_{t_k,\V w_k} \in \Spc P_{m-1}(\R^d)$, and the Radon radial-basis function $\rho_{\mathrm{Rad},m}\colon \R \to \R$ defined by \eqref{eq:RadonBasis}. 
	\end{theorem}

	\begin{remark}
		Starting from the chosen representative, we can also add the minimization over $\Spc P_{m-1}(\R^d)$ to the  problem
		 and replace $\Op P_{\mathrm{Rad},m}\{f\}(\V x_m)$ with $\Op P_{\mathrm{Rad},m}\{f\}(\V x_m) + p(\V x_m)$, where $p\in \Spc P_{m-1}(\R^d)$, which results in a minimization over $L_{\infty,m-1}(\R^d)$.
		As this rules out the dependence on the representation operator $P_{\mathrm{Rad},m}$, we are back in a classical setting and a similar result holds (with $K_0 \leq M-m$), see \cite[Thm.\ 3]{Unser2022}.
		Further, we can also evaluate $\|\partial_t^m \Op K_\mathrm{rad}\Op R\{f\}\|_{\Spc M_{m}}$ in the sense of $\Spc S^\prime(\R^d)$ since $\Spc P_{m-1}(\R^d) \subset \ker \partial_t^m \Op K_\mathrm{rad}\Op R$.
		Compared to previous results in the literature \cite{Bartolucci2021,Parhi2021}, this leads to a stronger characterization of the solution set $S$ together with a nice and elegant proof.
	\end{remark}

\section{Conclusions}\label{sec:Conclusions}
We have shown that continuous projections onto the Lizorkin space cannot exist.
Therefore, we had to resort to projection-free approaches to find representatives of Lizorkin distributions.
Using the property that the space is dense in $C_0(\R^d)$, we have established a framework for finding representatives of distributions that lie in certain Banach subspaces.
To do so, we only require representations of the related Green functions with sufficient regularity. 
Based on the obtained representation operator, we have introduced a powerful variational framework for the study of a wide class of inverse problems.
In particular, this enabled us to strengthen results obtained in prior works.
As future work, we want to apply our framework to the study of other subspaces and related variational models.

\section*{Acknowledgments}
The research leading to these results has received funding from the European Research Council (ERC) under European Union’s Horizon 2020 (H2020), Grant Agreement - Project No 101020573 FunLearn.
Further, the authors want to thank Joachim Krieger and Marc Troyanov for fruitful discussions on the topic and, in particular,  Joachim Kirieger for providing us with a proof of the nonexistence of projections.

\appendix
\section{Fundamental Solutions of the Fractional Laplacian}
Given $\alpha \in \R$ and $d \in \N$ with $\alpha > d$ and $\alpha -d \notin \N$, we want to provide an estimate of the asymptotic behavior of $f_{\alpha,d} \colon \R^d \to \R$ with $f_{\alpha,d}(\V x)=\Vert \V x \Vert^{\alpha-d}$ and of its derivatives.
For this purpose, we need the following lemma.
\begin{lemma}\label{lem:help}
	For any $\V k \in \N^d$, it holds that $\partial^{\V k} \Vert \cdot \Vert = p_{\V k}/\Vert \cdot \Vert^{-1 + 2\vert \V k \vert}$ for some polynomial $p \in  \Spc P (\R^d)$ of order at most $\vert \V k \vert$.
\end{lemma}
\begin{proof}
	We proceed by induction. For $\V k=\mathbf 0$ the result is obviously true.
	Assume that the claim holds for any $\V k \in \N^d$ with $\vert \V k \vert \leq n$ and let $\V k \in \N^d$ with $\vert \V k \vert = n + 1$.
	For simplicity of notation, we assume that the derivative w.r.t.\ $\V x_1$ is included and define $\tilde{\V k} = \V k - \V e_1$.
	The induction assumption implies that
	\begin{align}
		\partial^{\V k} \Vert \V x \Vert &= \partial_{\V x_k}\partial^{\tilde{\V k}} \Vert \V x \Vert  = \partial_{\V x_k} \frac{ p_{\tilde{\V k}}(\V x)}{\Vert \V x \Vert^{-1 + 2\vert \tilde{\V k} \vert}} = \frac{ \partial_{\V x_k}  p_{\tilde{\V k}}(\V x)\Vert \V x \Vert^{-1 + 2\vert \tilde{\V k} \vert} - p_{\tilde{\V k} }(\V x)\V x_k \Vert \V x \Vert^{-3 + 2\vert \tilde{\V k} \vert}}{\Vert \V x \Vert^{(2(-1 + 2\vert \tilde{\V k} \vert)}}\notag\\
		&= \frac{ p_{\V k}(\V x)\Vert\V x \Vert^{-3 + 2\vert \tilde{\V k} \vert}}{\Vert \V x \Vert^{(2(-1 + 2\vert \tilde{\V k} \vert)}} = \frac{ p_{\V k}(\V x)}{\Vert \V x \Vert^{-1 + 2\vert {\V k} \vert}},
	\end{align}
	which concludes the proof.
\end{proof}
Lemma~\ref{lem:help} is going to let us prove the actual result.
\begin{proposition}\label{eq:FundGrowth}
	For any $\V k \in \N^d$ with $\vert \V k\vert \leq \lceil \alpha -d \rceil$ and $x \neq 0$, it holds that $\vert \partial^{\V k} f_{\alpha,d}(\V x) \vert \leq C \Vert \V x \Vert^{\alpha -d - \vert \V k \vert}$ .
\end{proposition}
\begin{proof}
	We proceed inducetively over $\lceil \alpha -d \rceil$.
	For $\lceil \alpha -d \rceil = 1$, Lemma~\ref{lem:help} implies that
	\begin{equation}
		\vert \partial_{\V x_k}  f_{\alpha,d}(\V x) \vert =  \vert (\alpha -d -1)f_{\alpha-1,d}(\V x) \partial_{\V x_k} \Vert \V x \Vert \vert \leq C \Vert \V x \Vert^{\alpha -d - 1}.
	\end{equation}
	If $k= 0$, there is nothing to show.
	Assume now that the results holds for $\lceil \alpha -d \rceil = n$ and let $\alpha, d$ be such that $\lceil \alpha -d \rceil = n+1$.
	Like in the proof of Lemma~\ref{lem:help}, we assume again that the derivative w.r.t.\ $\V x_1$ is included and define $\tilde{\V k} = \V k - \V e_1$.
	Then, using the Leibniz rule, we provide the estimate
	\begin{align}
		\vert \partial^{\V k}  f_{\alpha,d}(\V x) \vert  &= \vert \partial^{\tilde{\V k}} \partial_{\V x_1}  f_{\alpha,d}(\V x) \vert \leq \! C \bigl\vert \partial^{\tilde{\V k}} \bigl( f_{\alpha-1,d}(\V x) \partial_{\V x_1} \Vert \V x \Vert \bigr) \bigr \vert \leq \! C \sum_{\V i < \tilde{\V k}} \bigl\vert \partial^{\tilde{\V k} - \V i} f_{\alpha-1,d}(\V x) \partial^{\V i + \V e_1} \Vert \V x \Vert \bigr \vert\notag\\
		&\leq C \sum_{\V i < \tilde{\V k}} \Vert \V x \Vert^{\alpha - 1 - d - \vert \tilde{\V k} - \V i \vert} \Vert \V x \Vert^{- \vert \V i \vert} \leq  C \Vert \V x \Vert^{\alpha -d - \vert \V k \vert},
	\end{align}
	which concludes the proof.
\end{proof}

\section{Radon Transform}\label{sec:Radon}
Here, we recall some important properties of the Radon transform, for which an extensive overview is given in \cite{Helgason1999}.
The Radon transform is first described for Lizorkin functions and then extended to distributions by duality.
\paragraph{Classical Integral Formulation}
The Radon transform of $f\in L_1(\R^d)$ is defined as
\begin{align}
	\Op R\{ f\}(t, \boldsymbol \xi)
	&=\int_{\R^d}\delta(t-\boldsymbol \xi^\Top\V x)  f(\V x) \dint \V x,\quad (t,\boldsymbol \xi) \in \R \times \mathbb{S}^{d-1}. \label{Eq:Radon2}
\end{align}
Its adjoint is the back-projection $\Op R^\ast$, whose action on $g\colon \R \times \mathbb{S}^{d-1} \to \R$ is defined as
\begin{align}
	\Op R^\ast \{g\}(\V x)=\int_{\mathbb{S}^{d-1}} g(\underbrace{\boldsymbol \xi^\Top\V x}_{t}, \boldsymbol \xi)\dint \boldsymbol \xi, \quad\V x\in \R^d.
	\label{Eq:Backprojection}
\end{align}
Given the Fourier transform $\widehat f\coloneqq \Fourier\{f\}$ 
of  $f \in L_1(\R^d)$, we can calculate $ \Op R \{f\}(\cdot,\boldsymbol \xi_0)$ at given  $\boldsymbol \xi_0 \in \mathbb{S}^{d-1}$ through the relation
\begin{align}
	\Op R\{ f\}(t, \boldsymbol \xi_0)=\frac{1}{2 \pi} \int_{\R} \widehat f(\omega\boldsymbol \xi_0) \ee^{\jj \omega t} \dint \omega= \Fourier^{-1}\{  \widehat f(\cdot\boldsymbol \xi_0) \}(t),
	\label{Eq:CentralSliceTheo}
\end{align}
a property that is referred to as the {\em Fourier-slice theorem}.
The key property for analysis purposes is that the Radon transform is continuous and invertible if the spaces are chosen properly, see \cite{Helgason1965,Helgason1999,Ludwig1966} for details.
\begin{theorem}[Continuity and invertibility of the Radon transform on $\Spc S_\mathrm{Liz}(\R^d)$]
	\label{Theo:RadonS0}
	The Radon operators $\Op R \colon \Spc S_\mathrm{Liz}(\R^d) \to \Spc S_\mathrm{Liz,0}(\R \times \mathbb{S}^{d-1})$ and $\Op R^* \colon \Spc S_\mathrm{Liz,0}(\R \times \mathbb{S}^{d-1}) \to \Spc S_\mathrm{Liz}(\R^d)$ are bijective and continuous.
	Moreover, $\Op R^\ast \Op K_\mathrm{rad} \Op R=\Op K \Op R^\ast \Op R=\Op R^\ast \Op R\Op K =\Identity \mbox{ on }\Spc S_\mathrm{Liz}(\R^d)$ and $\Op K_\mathrm{rad} \Op R \Op R^\ast=\Identity \mbox{ on }\Spc S_{\mathrm{Liz},0}(\R \times \mathbb{S}^{d-1})$, where $\Op K=(\Op R^\ast \Op R)^{-1}=c_d(-\Delta)^{(d-1)/2}$ with $c_d=(2(2\pi)^{d-1})^{-1}$ is the so-called ``filtering'' operator 
	and where $ \Op K_\mathrm{rad}$ is an one-dimensional radial counterpart that acts along the Radon-domain variable $t$.
	These filtering operators are characterized by their frequency response 
	$\widehat K(\bw)=c_d\|\bw\|^{d-1}$ and
	$\widehat K_\mathrm{rad}(\omega)=c_d |\omega|^{d-1}$. \end{theorem}

As evidenced in \eqref{Eq:GreenLap}, the impulse response of the filtering operator $\Op K$ in Theorem~\ref{Theo:RadonS0} is proportional to $k_{-d+1,d}$,  which tells us that it asymptotically decays like $1/\|\V x\|^{2d -1}$ when $d$ is even, or is a power of the Laplacian (local operator) otherwise.
Further, we note that Theorem \ref{Theo:RadonS0} implies that $\Op R$ is actually a homeomorphism.

\paragraph{Distributional Extension}
\label{Sec:RadonDistributions}
This framework is extended to distributions by duality.
\begin{definition}
	\label{Def:GeneralizedRadon}
	The distribution $g=\Op R\{ f \}
	\in \Spc S_{\mathrm{Liz},0}'(\R \times \mathbb S^{d-1})$ is the  Radon transform of $f \in \Spc S_\mathrm{Liz}'(\R^d)$ if
	\begin{krad}
		\begin{align}
			\forall \phi \in \Spc S_{\mathrm{Liz},0}(\R \times \mathbb S^{d-1}):\quad \langle g,\phi \rangle_\mathrm{Rad}
			=\langle f, \Op R^\ast \{\phi\} \rangle.
			\label{Eq:RDist}
		\end{align}
	\end{krad}
	\begin{krad}
		Likewise, $\tilde g=\Op K\Op R\{ f \} \in \Spc S_{\mathrm{Liz},0}'(\R \times \mathbb S^{d-1})$ is the filtered projection  of $f \in \Spc S_\mathrm{Liz}'(\R^d)$ if 
		\begin{align}
			\forall \phi \in \Spc S_{\mathrm{Liz},0}(\R \times \mathbb S^{d-1}): \quad \langle \tilde g,\phi \rangle_\mathrm{Rad}=\langle f, \Op R^\ast\Op K_\mathrm{rad} \{\phi\} \rangle.
			\label{Eq:KRDist}
		\end{align}
	\end{krad}
	Finally, the backprojection $f=\Op R^\ast \{g\} \in \Spc S_{\mathrm Liz}'(\R^d)$ of $g\in \Spc S_{\mathrm{Liz},0}'(\R \times   \mathbb{S}^{d-1})$ is defined via
	\begin{align}
		\forall \varphi \in \Spc S_\mathrm{Liz}(\R^d): \quad \langle \Op R^\ast\{ g\}, \varphi \rangle
		=\langle g, \Op R \{\varphi\}\rangle_\mathrm{Rad}. \label{Eq:RadjDis}
	\end{align}
\end{definition}

Due to duality,the distributional extension of the Radon transform inherits most of the properties of the ``classical'' operator defined by \eqref{Eq:Radon2}. 
\begin{theorem}[Invertibility of the Radon transform on $\Spc S_\mathrm{Liz}'(\R^d)$]
	\label{Theo:InvertRadonDist}
	It holds that $\Op R^\ast \Op K_\mathrm{rad} \Op R=\Op K \Op R^\ast \Op R
	=\Identity$ on $\Spc S_{\mathrm{Liz}}'(\R^d)$. 
	Hence, the ``filtered-projection'' operator $\Op K_\mathrm{rad}\Op R\colon \Spc S_\mathrm{Liz}'(\R^d) \to \Spc S'_{\mathrm{Liz},0}(\R^d)$ is a homeomorphism with inverse $\Op R^\ast\colon \Spc S'_{\mathrm{Liz},0}(\R^d) \to \Spc S_\mathrm{Liz}'(\R^d)$.
\end{theorem}
The Fourier-slice theorem expressed by \eqref{Eq:CentralSliceTheo} yields a unique (Fourier-based) characterization of $\Op R \{f\}$. 
It
remains valid for tempered distributions whose generalized Fourier transforms can be identified as continuous functions of $\bw$. It is especially helpful 
when the underlying function or distribution is isotropic. 

An isotropic function $\rho_\mathrm {iso}\colon \R^d \to \R$ is characterized 
by its radial profile $\rho\colon \R_{\ge0} \to \R$, so that
$\rho_\mathrm {iso}(\V x)=\rho(\|\bx\|)$. 
The frequency-domain counterpart of this characterization is 
$\widehat \rho_\mathrm {iso}(\bw)=\widehat \rho_\mathrm{rad}(\|\bw\|)$ with radial frequency profile
\begin{equation}
	\widehat \rho_\mathrm{rad}(\omega)= \frac{(2\pi)^{d/2}}{|\omega|^{d/2-1} }\int_{0}^{+\infty} \rho(t) t^{d/2-1} J_{d/2-1}(\omega t) t \dint t,
\end{equation}
where $J_{\nu}$ is the Bessel function of the first kind of order $\nu$.
In Proposition \ref{prop:IsoLiz}, we characterize isotropic Lizorkin functions.
\begin{proposition}\label{prop:IsoLiz}
	Let $\varphi_\mathrm{iso} \in \Spc S(\R^d)$ be an isotropic test function.
	Then, $\varphi_\mathrm{iso} \in \Spc S_\mathrm{Liz}(\R^d)$ if and only if $\varphi_\mathrm{rad}(t) = \Op R\{\varphi_\mathrm{iso}\}(t,\boldsymbol \xi) \in \Spc S_\mathrm{Liz}(\R)$.
	
\end{proposition}
\begin{proof}
	Since $\varphi_\mathrm{iso}$ is isotropic, for any $\V k \in \N^d$, we have that
	\begin{align}
		c_\V k=\langle \V x^{\bk},\varphi_\mathrm{iso} \rangle&=\jj^{|\V k|}\partial^{\V k}\widehat \varphi_\mathrm{iso}(\V 0)= \jj^{k}\Dop^{k}\{\widehat \varphi_\mathrm{rad}\}(0)=
		c_k \mbox{ with } k=|\V k|.\
	\end{align}
	The last equality also implies that $c_k=\int_{\R} \varphi_\mathrm{rad}(t) t^k \dint t$, where $\varphi_\mathrm{rad}=\Fourier^{-1}\{\widehat \varphi_\mathrm{rad}\} (\cdot,\boldsymbol \xi) = \Op R\{\varphi\}(\cdot,\boldsymbol \xi)$ is the radial profile (by the Fourier-slice theorem). 
	This shows that,  indeed,
	\begin{equation}
		\varphi_\mathrm{rad} \in \Spc S_\mathrm{Liz}(\R)\Leftrightarrow \varphi_\mathrm{iso}\in \Spc S_\mathrm{Liz}(\R^d).
	\end{equation}
\end{proof}
Finally, we provide a result on how to compute the Radon transform of isotropic Lizorkin distributions.
\begin{proposition} [Radon transform of isotropic distributions] 
	\label{Prop:IsoRad}
	Let $\rho_\mathrm{iso}$ be an isotropic distribution whose radial frequency profile is $\widehat \rho_\mathrm{rad}(\omega)$. Then,
	\begin{align}
		\Op R\{\rho_\mathrm{iso}(\cdot -\V x_0)\}(t,\boldsymbol \xi)&=\rho_\mathrm{rad}(t-\boldsymbol \xi^\Top \V x_0) \\
		\Op K_\mathrm{rad}\Op R\{ \rho_\mathrm{iso}(\cdot -\V x_0)\}(t,\boldsymbol \xi\}&=\tilde \rho_\mathrm{rad}(t-\boldsymbol \xi^\Top \V x_0)\\
		\Op R\{\partial^{\V m}\rho_\mathrm{iso}\}(t,\boldsymbol \xi)&=\boldsymbol \xi^{\V m} \Op D^{|\V m|}\{\rho_\mathrm{rad}\}(t)
		\label{Eq:RadDeriv}
	\end{align} 
	with $\rho_\mathrm{rad}(t)=\Fourier^{-1}\{\widehat \rho_\mathrm{rad}\}(t)$ and 
	$\tilde\rho_\mathrm{rad}(t)=\tfrac{1}{2(2\pi)^{d-1}}\Fourier^{-1}\{|\cdot|^{d-1}\widehat\rho_\mathrm{rad}\}(t)$. 
	
\end{proposition}
\begin{proof} These identities are all direct consequences of the
	Fourier-slice theorem. For instance, by setting $\bw =\omega \boldsymbol \xi$ in the Fourier transform of $\partial^{\V m}\rho_{\mathrm iso}$, we get that
	\begin{align}
		\widehat{\partial^{\V m}\rho_{\mathrm iso}}(\omega \boldsymbol \xi)
		=(\jj \omega\boldsymbol \xi)^{\V m}\widehat \rho_\mathrm{rad}(\omega)=\boldsymbol \xi^{\V m} (\jj \omega)^{|\V m|} \widehat \rho_\mathrm{rad}(\omega),
	\end{align}
	which, upon taking the inverse 1D Fourier transform, yields \eqref{Eq:RadDeriv}. 
\end{proof}

\section{Extreme Points}\label{sec:Projection}
First, we recall the definition of extreme points.
\begin{definition}[Extreme points]
Let $C$ be a convex set in a Banach space $\Spc X$.
The extreme points of $C$ are the points 
$x\in C$ such that if there exist $x_1,x_2 \in C$ and $\theta \in (0,1)$ with $x= \theta x_1 + (1- \theta)x_2$, then it necessarily holds that $x_1 = x_2$.
The set of extreme points is denoted by $\mathrm{Ext}(C)$.

\end{definition}
\begin{proposition}[Isometric projections and extreme points]
\label{Prop:ExtremeProj}
Let $\Spc U$ be a closed subspace of the Banach space $(\Spc X,\|\cdot\|_\Spc X)$ with some corresponding continuous projection $\mathrm{Proj}_\Spc U\colon \Spc X \to \Spc U$.
Then, the following hold:
\begin{enumerate}
\item The unit ball in the Banach space $\Spc U=\mathrm{Proj}_{\Spc U}(\Spc X)$ satisfies
 \begin{equation}
 	B_{\Spc U}(1)\subseteq \mathrm{Proj}_{\Spc U}(B_{\Spc X}(1)) \subseteq B_{\Spc U}(\|\Proj_{\Spc U}\|),
 \end{equation}
 where $B_{\Spc U}(r)=\{x \in \Spc U: \|u\|_{\Spc U}\le r\}$ and $\|\Proj_{\Spc U}\|$ is the norm of the underlying projector. Consequently, $B_{\Spc U}(1)=\mathrm{Proj}_{\Spc U}(B_{\Spc X}(1))$ if and only if  $\|\Proj_{\Spc U}\|=1$.
\item 
Let $\tilde E=\{
\mathrm{Proj}_\Spc U\{e\}: e \in \mathrm{Ext}(B_\Spc X(1)) \} \backslash\{ 0\}$. If $\|\Proj_{\Spc U}\|=1$ and all
  $\tilde e \in \tilde E$ satisfy $\|\tilde e\|_{\Spc X}=1$, then  $B_\Spc U(1)$ is the closed convex hull of $\tilde E$ so that
$\mathrm{Ext}(B_\Spc U(1)) \subseteq \tilde E$.

\begin{extra}
Intersection of a bounded convex set with an affine hyperplane (codim 1): every extreme points of intersection is a convex combination of (at most) two original extreme points.
\end{extra}
\end{enumerate}
\end{proposition}
\begin{proof}
For the first statement, note that the unit ball in $\Spc U$ is $B_{\Spc U}(1)=B_{\Spc X}(1)\cap \Spc U$.
In particular, $u=\mathrm{Proj}_{\Spc U} \{u\}$ and $\|u\|_{\Spc X}\le 1$ for any $u \in B_{\Spc U}(1)$, which implies that $B_{\Spc U}(1) \subseteq \mathrm{Proj}_{\Spc U}(B_{\Spc X}(1))$.
Next, we recall that the norm of $\Proj_{\Spc U}\colon \Spc X \to \Spc U$ is given by
\begin{equation}
	\|\Proj_{\Spc U}\|=\sup_{x \in \Spc X \backslash \{0\}} \frac{\|\Proj_{\Spc U}\{x\}\|_{\Spc X}}{\|x\|_{\Spc X}}.
\end{equation}
Therefore, any $x \in B_{\Spc X}(1)$ satisfies
$\|\Proj_{\Spc U}\{x\}\|_{\Spc X}\le \|\Proj_{\Spc U}\| \, \|x\|_{\Spc X}\le \|\Proj_{\Spc U}\|$,
which implies that $\mathrm{Proj}_{\Spc U}(B_{\Spc X}(1)) \subseteq B_{\Spc U}(\|\Proj_{\Spc U}\|)$.

The Krein-Milman theorem ensures that $B_\Spc X(1)$ 
is the closed convex
hull of its extreme points $e_k \in E=\mathrm{Ext}(B_\Spc X(1))$, namely $B_\Spc X(1)=\mathrm{cch}E$.
Due to $\|\Proj_\Spc U\|=1$, it holds that $B_{\Spc U}(1)=\mathrm{Proj}_{\Spc U}(B_{\Spc X}(1))$.
Further, as $B_{\Spc U}(1)$ is convex, each $u=\mathrm{Proj}_{\Spc U}(\sum_{k=1}^K \theta_k e_k)= \sum_{k=1}^K \theta_k \tilde e_k$ with $\theta_k\ge 0$, $\sum_{k=1}^K \theta_k=1$, and 
$e_k \in E$ lies in $
B_{\Spc U}(1)$.
In other words, the convex hull of the $e_k$ maps onto the convex hull of the $\tilde e_k=\mathrm{Proj}_{\Spc U} e_k$ with $\mathrm{ch}\{\tilde e_k\}\subseteq B_\Spc U(1)$. Since $\mathrm{Proj}_{\Spc U}$ is continuous
and $B_\Spc X(1)$ is closed, the argument carries over to limits as well. Hence, it holds that
$B_\Spc U(1) = \Proj_\Spc U(\mathrm{cch}E)=\mathrm{cch}(\Proj_\Spc U(E))=\mathrm{cch}(\tilde E)$.
\begin{extra}
Now let us assume that $\|\Proj_\Spc U\|=r\ge 1$ with all $\tilde e_k$ such that $\|\tilde e_k\|_{\Spc X}=r$.
Since $B_\Spc U(1) \subseteq \Proj_\Spc U(\mathrm{cch}E)$, for any $u \in \partial B_\Spc U(1)$, there exists some $\theta_k$ such that $u=\sum_{k} \theta_k \tilde e_k$ with
$\|u\|=1\le \sum_{k} \theta_k \|\tilde e_k\|$.  ???
\end{extra}
\end{proof}

\begin{example}
	The projector $\Proj_{\mathrm even}\colon \Spc S(\R^d) \to  \Spc S_{\mathrm even}(\R^d)$ onto the even Schwartz functions is given by
	\begin{align}
		\Proj_{\mathrm even}\{f\}(\V x)= \frac{f(\V x) +f(-
			\V x)}{2}.
	\end{align}
	By duality, we define $\Proj_{\mathrm even}\colon \Spc S'(\R^d) \to  \Spc S'_{\mathrm even}(\R^d)$.
	The extreme points of $\Spc M(\R^d)$ are $(\delta(\cdot-\V \tau))_{\V \tau\in \R^d}$.
	Since
	\[\|\Proj_{\mathrm even}\{\delta(\cdot-\V \tau)\}\|_{\Spc M}= \|\tfrac{1}{2}\delta(\cdot+\V \tau) +\tfrac{1}{2}\delta(\cdot-\V \tau)\|_{\Spc M}=1\]{l}
	for all $\V \tau \in \R^d$, the extreme points of
	$\Spc M_{\mathrm even}(\R^d)$ are of the form $\tfrac{1}{2}\delta(\cdot+\V \tau) +\tfrac{1}{2}\delta(\cdot-\V \tau)$ with ${\V \tau \in \R^d}$.
\end{example}

\begin{extra}
\subsubsection{Duality maps}
The generalization of the Cauchy-Schwarz inequality for any dual pair of Banach spaces $(\Spc X, \Spc X')$ is
the so-called {\em duality inequality}
\begin{align}
\label{Eq:Dualitybound}
\forall (f,x) \in \Spc X' \times \Spc X: \langle f, x\rangle_{\Spc X' \times \Spc X}\le \|f\|_{\Spc X'}\|x\|_{\Spc X},
\end{align}
which is tightly linked to the definition of the dual norm: 
\begin{align}
\label{Eq:DualNorm}
\|f\|_{\Spc X'}\eqdef\sup_{x \in \Spc X: \|x\|_{\Spc X}\le 1} \langle f, x\rangle_{\Spc X' \times \Spc X}.
\end{align}
\begin{definition}[Strict convexity]
A Banach space $\Spc X$ 
is said to be strictly convex if, for all $x_1,x_2 \in \Spc X$ such that $\|x_1\|_{\Spc X} =\|x_2\|_{\Spc X} =1$
and $x_1 \ne x_2$, one has that $\|\theta x_1+(1-\theta)x_2\|_{\Spc X}< 1$ for all
$\theta \in (0,1)$.
\end{definition}

It is obvious from \eqref{Eq:DualNorm} that \eqref{Eq:Dualitybound} is sharp. Moreover, when $\Spc X$ 
is reflexive and strictly convex, there is a single element $x^\ast \in \Spc X'$ (the Banach conjugate of $x$)
such that $\|x^\ast\|_{\Spc X'}=\|x\|_{\Spc X}$ (isometry) and
$\langle x^\ast, x\rangle_{\Spc X' \times \Spc X}=\|x^\ast\|_{\Spc X'}\|x\|_{\Spc X}$ (sharp duality bound) \cite{Cioranescu2012}. This leads to the definition of the corresponding duality map $\Op J_\Spc X: \Spc X \to \Spc X'$:
\begin{align}
\Op J_{\Spc X}\{x\}=x^\ast.
\end{align}
Since the dual of $\Spc X'$ is strictly convex as well, we have that $\Op J_\Spc X^{-1}=\Op J_{\Spc X'}: \Spc X' \to \Spc X$ with $\Op J_{\Spc X'}\{x^\ast\}=x$ where $(x^\ast)^\ast=x \in \Spc X''=\Spc X$
is the unique Banach conjugate of $x^\ast \in \Spc X'$.

A relevant example of reflexive and strictly convex Banach space is $\Spc X=L_q(\R \times \mathbb{S}^{d-1})$ for $q \in (1,\infty)$.
Its topological dual is
$\Spc X'=L_p(\R \times \mathbb{S}^{d-1})$ with $p=q/(q-1)$ the conjugate exponent of $q$.
For that particular pair,
\eqref{Eq:Dualitybound} reduces to the H\"older inequality for hyperspherical functions.
The corresponding duality map $\Op J_q: L_q(\R \times \mathbb{S}^{d-1}) \to L_p(\R \times \mathbb{S}^{d-1})$
is (see \cite[Chap.\ 4]{Cioranescu2012})
\begin{align}
\label{Eq:LpDual}
\Op J_{q}\{\nu\}(\V z)=\nu^\ast(\V z)=\frac{\left|\nu(\V z)\right|^{q-1}}{\|\nu\|^{q-2}_{L_q}} \mathrm{sign}\big(\nu(\V z)\big).
\end{align}
The latter establishes a one-to-one isometric mapping between the spaces $L_q$ and 
$L_p=(L_q)'$ with the property that $\Op J^{-1}_{q}=\Op J_{p}$.

\begin{proposition}[Banach isometries]
\label{Prop:DualMaps}
Let $(\Spc X, \Spc X')$ be a dual pair of reflexive and strictly convex Banach spaces with corresponding duality map $\Op J_\Spc X: \Spc X \to \Spc X'$. We consider two generic types of linear isometries:

1) One-to-one map: Let $\Op T: \Spc X \to \Spc Y=\Op T(\Spc X)$ be an injective operator whose inverse is denoted by $\Op T^{-1}$ with $\Op T^{-1}\Op T=\Identity$ on $\Spc X$. Then, $\Spc Y=\Op T(\Spc X)=\{y=\Op T\{x\}: x \in \Spc X\}$ equipped with the norm $\|y\|_{\Spc Y}=\|\Op T^{-1}\{y\}\|_{\Spc X}$ is a Banach space. Its 
continuous dual is the Banach space $\Spc Y'=\Op T^{-1\ast}(\Spc X')$ with
$\|y'\|_{\Spc Y'}=\| \Op T^\ast\{ y'\}\|_{\Spc X'}$, while the
corresponding duality map is
$\Op J_{\Spc Y}= (\Op T^\ast)^{-1}\circ \Op J_{\Spc X}\circ\Op T^{-1}: \Spc Y \to \Spc Y'$.

2) Projection: Let $\Op P: \Spc X \to \Spc U \embedIso \Spc X$ be a continuous projection on $\Spc X$ with $\|\Op P\|=1$. Then, $\big(\Spc U=\Op P(\Spc X),\Spc U'=\Op P^\ast(\Spc X')\big)$ is a dual pair of Banach subspaces with corresponding duality map
$\Op J_{\Spc U}=\Op P^\ast\circ \Op J_{\Spc X}\circ\Op P: \Spc U \to \Spc U'$.

\end{proposition}
\begin{proof}
We first recall that the dual of a reflexive and strictly convex Banach space is reflexive (by definition) and strictly convex as well.
\item 1) Injective operator: 
For the first property, we refer to \cite[Prop.\ 1]{Unser2022}. The key observation is that the operators $\Op T: \Spc X \to \Spc Y$ and 
$\Op T^{-1}: \Spc Y \to \Spc X$, as well as their adjoints are isometries with
$(\Op T^\ast)^{-1}=\Op T^{-1\ast}$. The argument then primarily rests upon the following duality inequality 
\begin{align}
\langle y', y\rangle_{\Spc Y' \times \Spc Y}=\langle y', \Op T\Op T^{-1}\{y\}\rangle_{\Spc Y' \times \Spc Y}&=\langle \Op T^\ast\{y'\}, \Op T^{-1}\{y\}\rangle_{\Spc X' \times \Spc X}\nonumber \\
&\le \|\Op T^\ast\{y'\}\|_{\Spc X'}\; \|\Op T^{-1}\{y\}\|_{\Spc X},
\end{align}
which is sharp if and only if $x=\Op T^{-1}\{y\}$ and $x'=\Op T^\ast\{y'\}$ (resp., $y'$ and $y$)
are Banach conjugates; that is, $x'=\Op J_{\Spc X}\{x\}$.

\item 2) Projection operator. The first part is covered by Proposition~\ref{Prop:ExtremeProj}.
For the second part, let $u \in \Spc U\embedIso \Spc X$ with Banach conjugate
$u^\ast \in \Spc X'$. Then,
\begin{align}
\|u^\ast\|_{\Spc X'}\|u\|_{\Spc X}
=\langle u^\ast, u\rangle_{\Spc X' \times \Spc X} & =
\langle u^\ast, \Op P^2 u\rangle_{\Spc X' \times \Spc X}=\langle \Op P^\ast u^\ast, \Op P \Op u\rangle_{\Spc U' \times \Spc U}\nonumber\\
& \ \le \|\Op P^\ast\| \|u^\ast\|_{\Spc U'} \|\Op P\| \|u\|_{\Spc U},
\end{align}
from which  we deduce that $\langle \Op P^\ast u^\ast, \Op P \Op u\rangle_{\Spc U' \times \Spc U}=\|u^\ast\|_{\Spc U'} \|u\|_{\Spc U'}$; that is, that $u=\Op P u$ and $u^\ast=\Op P^\ast u^\ast=
\Op P^\ast \circ \Op J_{\Spc X}\{\Op P u\}$ are (unique) Banach conjugates of each other.
\end{proof}
\end{extra}

\bibliographystyle{abbrv}
\bibliography{Lizorkin.bib}

\end{document}